\documentclass[a4paper,fleqn]{cas-sc}

    \usepackage[numbers,sort&compress]{natbib}

\usepackage[utf8]{inputenc} 
\usepackage{accents}				
\usepackage{todonotes}
\usepackage{subcaption}
\usepackage{stmaryrd} 
\usepackage{amsthm} 
\usepackage[makeroom]{cancel} 
\usepackage{bookmark} 
\usepackage{empheq}
\usepackage{ulem}
\usepackage{booktabs} 
\usepackage{siunitx}
\usepackage[shortcuts]{extdash}

\newcommand{\red}[1]{{\leavevmode\color{red}#1}}

\newcommand{\norm}[1]{\left\lVert#1\right\rVert}
\newcommand{\mat}[1]{\underline{\mathbf{#1}}}
\newcommand\acclrvec[1]{\accentset{\,\leftrightarrow}{#1}}	
\newcommand{\blocktensor}[1]{\acclrvec{{\mathbf #1}}}	
\newcommand\blocktensorG[1]{\acclrvec{\boldsymbol #1}}	
\newcommand{\Nabla} {\vec{\nabla}}
\newcommand{\numfluxb}[1]{\hat{\mathbf{#1}} }
\newcommand{\numflux}[1]{\hat{{#1}} }
				
\newcommand{\ec}{{\mathrm{EC}}}			
	
\def\polydeg{N}

\newcommand{\diag}{\mathrm{diag}}

\newcommand{\bigpartialderiv}[2]{ \frac{\partial {#1}}{\partial {#2} } }

\newcommand\stateG[1]{\boldsymbol #1}

\newcommand\state[1]{\mathbf{#1}}
\newcommand{\contSt}[1]{\tilde{\state{#1}}}

\newcommand{\elemDom}{K}	

\newcommand{\avg}[1]{\left\{\hspace*{-3pt}\left\{#1\right\}\hspace*{-3pt}\right\}}

\newcommand{\jump}[1]{\ensuremath{\left\llbracket #1 \right\rrbracket}}

\newtheorem{remark}{Remark}
\newtheorem{lemma}{Lemma}
\newtheorem{corollary}[lemma]{Corollary}

\usepackage{xcolor}
\usepackage[dvipsnames]{xcolor}

\usepackage{soul}

\begin{document}

\newcommand{\diffcomment}[1]{%
    \ifdefined\DIFadd
        {\color{orange} #1}
    \fi
}

\let\WriteBookmarks\relax
\def\floatpagepagefraction{1}
\def\textpagefraction{.001}

\shorttitle{GPU-Accelerated Nested Schur Complement for Atmospheric HEVI-DG}
\shortauthors{Artiano et al.}

\title [mode = title]{GPU-Accelerated Nested Schur-Complement Solvers for Atmospheric Flow Simulations Using Horizontally-Explicit, Vertically-Implicit Schemes and Discontinuous Galerkin Methods}

\author[1]{Marco Artiano}[orcid=0009-0009-5872-702X]
\cormark[1]
\ead{martiano@uni-mainz.de}
\cortext[cor1]{Corresponding author}

\author[2]{Andrés M. Rueda-Ramírez}[orcid=0000-0001-6557-9162]

\author[3]{Oswald Knoth}

\author[1]{Hendrik Ranocha}[orcid=0000-0002-3456-2277]

\author[4,5]{Gregor Gassner}[orcid=0000-0002-1752-1158]

\address[1]{Institute of Mathematics, Johannes Gutenberg University Mainz, Mainz, Germany}

\address[2]{ETSIAE-UPM, School of Aeronautics, Universidad Politécnica de Madrid, Madrid, Spain}

\address[3]{Leibniz Institute for Tropospheric Research (TROPOS), Leipzig, Germany}

\address[4]{Department of Mathematics and Computer Science, University of Cologne, 
Cologne, Germany}

\address[5]{Center for Data and Simulation Science, University of Cologne, 
Cologne, Germany}




\begin{abstract}
Horizontally explicit, vertically implicit (HEVI) time integration removes the acoustic time-step restriction imposed by the fine vertical resolution of nonhydrostatic atmospheric models, but its efficiency depends on the fast solution of the resulting vertical linear systems, particularly on GPU-accelerated hardware.
In this work, we present GPU-accelerated nested Schur-complement solvers for these systems for a split-form discontinuous Galerkin spectral element method on Legendre--Gauss--Lobatto (LGL) nodes applied to the compressible Euler equations with gravity.

Time integration is performed with a Rosenbrock-W method, a linearly implicit Runge--Kutta scheme that allows the use of approximate Jacobians while maintaining high-order accuracy and robust stability properties for stiff atmospheric wave dynamics.
We exploit this freedom to construct a sparse vertical Jacobian that captures vertically propagating acoustic waves and gravity on curvilinear meshes, including orography, and that decouples into independent systems for each vertical column.

The solvers are \emph{nested}: a global Schur complement, known as static condensation, exploits the sparsity of the discretization on LGL nodes to condense each column system onto a small set of element-edge unknowns, and a local Schur complement eliminates, within each element, all prognostic variables but the vertical momentum.
The local elimination benefits from the sparsity that arises naturally when the thermodynamic equation is written in terms of potential temperature, since the pressure then depends on a single prognostic variable.
We propose two global condensation strategies, which assign the edge unknowns differently and lead to a banded and a block-tridiagonal reduced system, respectively, and analyze their structure and cost.
The resulting HEVI schemes conserve the total mass and the total potential-temperature  exactly up to round-off errors.
The solvers are implemented for CPUs and GPUs in TrixiAtmo.jl, the atmospheric extension of Trixi.jl.

We verify the high-order convergence of the method with linear inertia--gravity waves in a two-dimensional channel, and we demonstrate its robustness on the sphere with standard test cases: the baroclinic instability, a mountain-induced baroclinic instability, and the gap-flow and vortex-shedding cases of DCMIP-2025.
For the baroclinic instability on a sphere without orography, the scheme achieves $18.8$ and $2.9$ simulated years per day at $\SI{100}{km}$ and $\SI{50}{km}$ horizontal resolution, respectively, on a single NVIDIA GH200 GPU in double precision.

\end{abstract}




\maketitle

\section{Introduction}\label{sec:intro}

Numerical weather prediction and climate models are moving toward kilometer-scale global resolutions, which have only recently become feasible on GPU-accelerated supercomputers \cite{Giorgetta2022,Hohenegger2023,Taylor2023,Donahue2024,lapillonne2026operational,stevens2019dyamond,satoh2019global}.
At these scales, the hydrostatic approximation is no longer valid and the fully compressible, nonhydrostatic Euler equations must be solved \cite{satoh2019global,stevens2019dyamond}.
The vertical grid spacing of such models is typically one to two orders of magnitude smaller than the horizontal one, so the time step of a fully explicit scheme is dictated by vertically propagating acoustic waves, which are of little meteorological relevance.
Horizontally explicit, vertically implicit (HEVI) schemes remove this restriction by treating only the terms that are responsible for the vertical propagation of acoustic waves implicitly \cite{ikawa1988comparison,skamarock2008time,satoh2008nicam,skamarock2012multiscale,zaengl2015icon,giraldo2013implicit,gardner2018imex}.
The implicit problem then decouples into independent one-dimensional systems, one for each vertical column, while the time step of the explicit part is limited only by the horizontal grid spacing \cite{giraldo2013implicit,abdi2019acceleration,vogl2019evaluation,reddy2023schur}.
In this paper, we present nested Schur-complement solvers for these column systems, which exploit the structure of high-order discontinuous Galerkin (DG) discretizations on Legendre--Gauss--Lobatto (LGL) nodes and the freedom in the choice of the Jacobian offered by Rosenbrock-W time integration methods, and which are designed for GPU architectures.

High-order element-based Galerkin methods have shown clear promise for atmospheric dynamical cores.
Prominent examples are the continuous and discontinuous Galerkin formulations of the nonhydrostatic unified model of the atmosphere (NUMA) \cite{GiraldoRestelli2008,giraldo2013implicit}, the spectral-element dynamical core CAM-SE/HOMME \cite{Dennis2012}, whose GPU implementation in the Simple Cloud-Resolving E3SM Atmosphere Model (SCREAM) has reached cloud-resolving resolutions on an exascale system \cite{Taylor2023,Donahue2024}, and the DG-based ClimateMachine \cite{Sridhar2022climatemachine}.
DG methods, in particular, combine high-order accuracy on curvilinear meshes with a compact stencil, which makes them well suited for massively parallel and GPU architectures \cite{abdi2019acceleration,candelaresi2026massivelyparallelnumericalsimulations,ntoukas2026horses3d,waterhouse2026gpu}.

Among DG methods, the DG spectral element method (DGSEM) on LGL nodes is a particularly efficient variant, which uses the same LGL nodes for interpolation and numerical integration \cite{black1999conservative,kopriva2009implementing}. 
The resulting differentiation operators satisfy the summation-by-parts (SBP) property \cite{gassner2013skew}, which allows the construction of so-called split-form (or flux-differencing) DGSEM schemes with structure-preserving nonlinear stability properties.
For the Euler equations, these include entropy stability (i.e., the discrete fulfillment of the second law of thermodynamics) \cite{fisher2013high,chandrashekar2013kinetic,ranocha2018generalised}, kinetic energy preservation \cite{gassner2016split,chandrashekar2013kinetic,ranocha2018generalised,ranocha2020entropy}, and pressure equilibrium preservation \cite{shima2021preventing,ranocha2022preventing}, among others.
These nonlinear stability guarantees considerably improve the robustness of high-order DG methods for under-resolved flows \cite{ChanRanochaRuedaRamirez2022,gassner2016split,shima2021preventing,rueda2021subcell,ranocha2018generalised}, and the resulting schemes can be implemented efficiently \cite{Ranocha2023,waterhouse2026gpu}.
For atmospheric flows, an additional requirement is the preservation of hydrostatic equilibria, which has motivated the development of well-balanced DG methods for the Euler equations with gravity \cite{li2016well,ChandrashekarZenk2017,WuXing2021}.
\citet{waruszewski2022entropy} combined entropy stability and well-balancedness in split-form DG discretizations of the Euler equations with gravity, and \citet{artiano2025structure,artiano2026entropy} recently extended these ideas to the potential-temperature formulation and to the vector-invariant form of the Euler equations, {as well as well-balanced schemes on arbitrary curvilinear meshes~\cite{artiano2025structure}.}
Related structure-preserving SBP discretizations have been developed for the shallow water equations, both in planar geometries \cite{wintermeyer2017entropy,ersing2024entropy,ersing2024hydrostatic} and on the sphere \cite{RicardoDuruLee2024b,Montoya2026}, and for the moist compressible Euler equations \cite{RicardoLeeDuru2024c}.
Compatible finite element methods achieve related properties with a different approach, based on finite element exterior calculus \cite{cotter2023compatible,melvin2019mixed}.

HEVI time integration has a long tradition in nonhydrostatic atmospheric models.
\citet{ullrich2012operator} integrated the vertical acoustic modes implicitly with an operator-split Runge--Kutta--Rosenbrock scheme, and \citet{weller2013runge} and \citet{colavolpe2017rk} analyzed implicit-explicit (IMEX) Runge--Kutta methods for HEVI splittings.
\citet{gardner2018imex} compared several IMEX additive Runge--Kutta methods and splittings in a global nonhydrostatic model and found that treating horizontal dynamics implicitly does not pay off due to the cost of the globally coupled systems, whereas the HEVI splitting that treats all vertical dynamics implicitly performed best.
For the nonhydrostatic spectral-element dynamical core HOMME-NH, \citet{vogl2019evaluation} evaluated a large number of IMEX Runge--Kutta methods, and \citet{SteyerVoglTaylorGuba2019} derived new methods with large stability regions for HEVI splittings.
\citet{lee2021exact} constructed a HEVI scheme with an exact balance of energy exchanges.
In the context of element-based Galerkin methods, \citet{giraldo2013implicit} analyzed IMEX formulations of NUMA, including a one-dimensional vertically implicit variant, and \citet{abdi2019acceleration} accelerated these formulations on manycore processors.
\citet{KoperaGiraldo2014} combined IMEX DG discretizations with adaptive mesh refinement, \citet{baldauf2021horizontally} developed a HEVI DG scheme for terrain-following coordinates, and \citet{reddy2023schur} derived Schur-complement IMEX formulations for DG discretizations.
Multirate schemes, such as the multirate infinitesimal step methods of \citet{wensch2009multirate}, offer an alternative that integrates the fast modes with smaller time steps.

An alternative to IMEX Runge--Kutta methods are the linearly implicit methods introduced by \citet{rosenbrock1963some} in 1963, which require the solution of a single linear system per stage instead of a nonlinear system \cite{kaps1979generalized,hairer1996solving,lang2021rosenbrock}.
\citet{steihaug1979attempt} generalized these methods to the so-called W-methods, whose order conditions hold for any approximation of the Jacobian \cite{norsett1979order}.
Rosenbrock-type schemes have been used in atmospheric models, e.g., in the operator-split scheme of \citet{ullrich2012operator}, in the All Scale Atmospheric Model (ASAM) of \citet{jahn2015asam}, where a linearly implicit Rosenbrock method ensures stability in the presence of fast sound waves, and by \citet{lee2024comparison}, who reported improved energetics and stability with respect to a Crank--Nicolson semi-implicit integrator.
They have also been combined with high-order DG discretizations of the compressible Navier--Stokes equations by \citet{bassi2015linearly} and \citet{liu2016comparative}.
Since the approximate Jacobian of a W-method only affects the stability and efficiency of the scheme, it can be designed such that the resulting linear systems admit a fast solver, which is the path we follow in this work.

In HEVI schemes, the cost of the implicit part is dominated by the solution of the column systems, which have to be solved in every stage of every time step.
For DG discretizations, each column system is block tridiagonal, with one dense block per element and off-diagonal blocks that couple neighboring elements through the numerical fluxes, and it is commonly solved with banded or block-tridiagonal direct solvers, e.g., \cite{giraldo2013implicit}.
Two families of Schur-complement techniques can further reduce this cost.
The first one is static condensation, introduced by \citet{guyan1965reduction} in 1965 for continuous finite elements, which eliminates the element-interior unknowns locally and leaves a globally coupled system for the element-edge unknowns only; for DG methods, it is closely related to hybridized discontinuous Galerkin (HDG) methods \cite{cockburn2009unified,nguyen2009implicit,nguyen2009nonlinear,nguyen2011implicit}.
\citet{rueda2021statically} showed that the DGSEM on LGL nodes can be statically condensed efficiently, since its element-interior unknowns couple to the neighboring elements only through the element-edge nodes, and \citet{laskowski2020advantages} quantified the resulting advantages for implicit solvers of the compressible Navier--Stokes equations.
More recently, \citet{pan2025half} proposed half-closed DG discretizations, whose condensation resembles one of the strategies presented in this work, but without a local elimination of prognostic variables.
The second family is used in semi-implicit atmospheric models, where Schur complements eliminate selected prognostic variables from the linearized system.
\citet{giraldo2010semi} derived several such formulations and showed that the choice of prognostic variables determines the structure of the resulting systems, and \citet{reddy2023schur} combined this approach with DG discretizations, reducing the effective system size to approximately one fifth of the original problem.
To the authors' knowledge, both families of techniques have not yet been combined for the HEVI time integration of DG discretizations.

In this paper, we combine both Schur-complement techniques into \emph{nested} Schur-complement solvers for the Rosenbrock\-/W HEVI time integration of a split-form LGL-DGSEM for the compressible Euler equations with gravity.
We first derive an approximate vertical Jacobian in closed block form, which captures the vertically propagating acoustic waves and gravity terms on curvilinear meshes.
By projecting the momentum onto the unit normal of the coordinate surfaces, we obtain an acoustic subsystem that is free of geometric (Christoffel) terms on arbitrary curvilinear meshes, including meshes with orography.
The resulting solvers combine an analytical, element-local elimination of density and potential-temperature density with a global static condensation of each column onto element-edge unknowns.
We propose two condensation strategies, which we call \emph{two-sided} and \emph{one-sided}, whose reduced systems are banded and block tridiagonal, respectively, and we analyze their structure and computational cost.
Both strategies are implemented for CPUs and GPUs in TrixiAtmo.jl \cite{trixiatmo} using KernelAbstractions.jl \cite{Churavy_KernelAbstractions_jl}, which is vendor agnostic and supports both NVIDIA and AMD hardware.
We verify the high-order convergence of the method with linear inertia--gravity waves in a two-dimensional channel, and we demonstrate its robustness on the sphere with standard test cases: the baroclinic instability, a mountain-induced baroclinic instability, and the gap-flow and vortex-shedding cases of DCMIP-2025.
For the baroclinic instability on a sphere without orography, the scheme achieves $18.8$ and $2.9$ simulated years per day at $\SI{100}{km}$ and $\SI{50}{km}$ horizontal resolution, respectively, on a single NVIDIA GH200 GPU in double precision.

The paper is organized as follows.
Section~\ref{sec:equations} presents the governing equations.
In Section~\ref{sec:hevidgsem}, we describe the split-form LGL-DGSEM and the Rosenbrock-W methods, and we derive the vertical Jacobian.
Section~\ref{sec:nestedschur} introduces the nested Schur-complement solvers and the two condensation strategies, and Section~\ref{sec:implementation} describes their implementation on CPUs and GPUs.
The solvers are tested in Section~\ref{sec:results}, and we draw our conclusions in Section~\ref{sec:Conclusions}.

\section{Governing equations}\label{sec:equations}
We consider the compressible Euler equations with gravity and Coriolis forces, written with the potential-temperature density as the thermodynamic prognostic variable,
\begin{equation}\label{eq:EulerPotentialTemperature0}
	\begin{aligned}
		 & \partial_t \varrho + \Nabla \cdot \left (\varrho \vec{v} \right ) = 0,                                         \\
		 & \partial_t (\varrho \vec{v}) + \Nabla \cdot ( \varrho \vec{v} \otimes \vec{v}) + \Nabla p = -\varrho \Nabla \phi - 2 \vec{\Omega} \times \varrho \vec{v}, \\
		 & \partial_t (\varrho \theta) + \Nabla \cdot \left (\varrho \theta \vec{v} \right ) = 0,
	\end{aligned}
\end{equation}
where $\varrho$ is the density, $\vec{v} = (u,v,w)^T$ the velocity in Cartesian coordinates, $\theta$ the potential temperature, $\phi = \phi(\vec x)$ the geopotential, and $\vec{\Omega} = (0, 0, \Omega)^T$ the angular velocity of the Earth, with the rotation axis aligned with the $z$-axis.
The potential temperature is related to the temperature $T$ through $\theta = T\,(p_0/p)^{R/c_p}$, where $c_p$ and $c_v$ are the specific heat capacities at constant pressure and volume, $R = c_p - c_v$ is the gas constant, and $p_0$ is a constant reference pressure.
Combined with the ideal-gas law $p = \varrho R T$, this yields the equation of state
\begin{equation}\label{eq:eos}
	p = p_0 \left( \frac{R\, \varrho \theta}{p_0} \right)^{\gamma},
\end{equation}
with the heat capacity ratio $\gamma = c_p / c_v$.
Throughout this work, we take $\Omega  = \SI{7.292e-5}{rad/s}$, $p_0 = \SI{100000}{Pa}$, $R = \SI{287}{J \, kg^{-1} K^{-1}}$, $c_p = \SI{1004}{J \, kg^{-1} K^{-1}}$, $c_v = \SI{717}{J \, kg^{-1} K^{-1}}$.

The equation of state \eqref{eq:eos} shows that, in this formulation, the pressure depends on a single prognostic variable, the potential-temperature density $\varrho\theta$.
As we will see in Section~\ref{sec:jacobian_hevi}, this property leads to a particularly sparse structure of the vertical Jacobian.

We use the compact notation
\begin{equation}\label{eq:EulerPotentialTemperature}
    \partial_t \state{u} + \Nabla \cdot \blocktensor{f} = \state{s},
\end{equation}
with the state vector of prognostic quantities $\state{u} = (\varrho, \varrho \vec{v}, \varrho \theta)^T$, the flux $\blocktensor{f} = (\state{f}_1, \state{f}_2, \state{f}_3)^T$,
\begin{equation}
\blocktensor{f} =
    \begin{pmatrix}
        {\varrho}\, {\vec v} \\[2pt]
        {\varrho}\, {\vec v} \otimes {\vec v} + {p}\, \mat I \\[2pt]
        {\varrho}\, \theta\, {\vec v}
\end{pmatrix},
\end{equation}
and the source term
\begin{equation}
	\state{s} =
	\begin{pmatrix}
		0 \\[2pt]
		-\varrho \Nabla \phi - 2 \vec{\Omega} \times \varrho \vec{v} \\[2pt]
		0
	\end{pmatrix}.
\end{equation}
In the split-form discretization of Section~\ref{sec:dgsem}, the gravity term $-\varrho\Nabla\phi$ is treated as a non-conservative product rather than as a source term, which is key to obtaining a well-balanced scheme \cite{waruszewski2022entropy,artiano2025structure}. 

\section{HEVI-DGSEM}\label{sec:hevidgsem}

In this section, we detail the high-order numerical discretization.
First, we describe the structure-preserving split-form DGSEM discretization of the governing equations in Section~\ref{sec:dgsem}.
Second, we introduce Rosenbrock-W methods and detail the HEVI splitting of the right-hand-side operator in Section~\ref{sec:rosenbrock}.

\subsection{LGL-DGSEM}
\label{sec:dgsem}

We discretize \eqref{eq:EulerPotentialTemperature} with LGL-DGSEM.
The computational domain is partitioned into non-overlapping, possibly curvilinear, hexahedral elements $\elemDom$, each obtained as the image of the reference cube $[-1,1]^3$ under a mapping $\vec x(\xi^1,\xi^2,\xi^3)$. 
Within each element, the solution and the flux are approximated by tensor-product Lagrange interpolating polynomials of degree $\polydeg$ through the $M=\polydeg+1$ LGL nodes $\{\xi_i\}_{i=0}^{\polydeg}$ on $[-1,1]$, with associated quadrature weights $\{\omega_i\}_{i=0}^\polydeg$; we index nodes $i,j,k \in \{0,\dots,\polydeg\}$ in the three reference directions and write $\state u_{ijk}(t) \approx \state u(\vec x_{ijk},t)$ for the nodal solution.

\paragraph{Mapping and metric terms.}
The mapping Jacobian is denoted $J = \det\left(\partial \vec x/\partial(\xi^1,\xi^2,\xi^3)\right)$, and the contravariant basis vectors $\vec a^m = \Nabla \xi^m$, $m=1,2,3$, define the transformation between reference and physical space. Following \citet{kopriva2006metric}, the (scaled) contravariant metric terms $(J\vec a^m)_{ijk}$ are computed via the curl form of the metric identities rather than the direct cross-product form, which is required for the discrete metric identities to hold exactly in three dimensions and is what guarantees that the scheme below preserves a constant (free-stream) solution on curvilinear meshes. The physical flux $\blocktensor f = (\state f_1, \state f_2, \state f_3)^T$ of \eqref{eq:EulerPotentialTemperature} is then transformed into the three contravariant fluxes
\begin{equation}
	\contSt f^m_{ijk} := \blocktensor f(\state u_{ijk}) \cdot (J\vec a^m)_{ijk}, \qquad m=1,2,3,
\end{equation}
and the outward unit normal at each of the six element faces is recovered from the same metric terms, e.g., at the face $i=\polydeg$, $\vec n_{\polydeg jk} = (J\vec a^1)_{\polydeg jk}/\norm{J\vec a^1}_{\polydeg jk}$ (and similarly, with a sign change, at $i=0$, and for the $j$- and $k$-faces).

\paragraph{One-dimensional SBP operators.}
Collocating the LGL interpolation and quadrature nodes gives the (dense) polynomial differentiation matrix $D_{jk} = \ell_k'(\xi_j)$ and the diagonal LGL mass matrix $\mat M = \diag(\omega_0,\dots,\omega_{\polydeg })$, related through the summation-by-parts (SBP) property
\begin{equation}
	\mat Q + \mat Q^T = \mat B, \qquad \mat Q \coloneqq \mat M \mat D, \qquad \mat B \coloneqq \diag(-1,0,\dots,0,1).
\end{equation}
$\mat B$ is the boundary evaluation matrix. 
Two combinations of $\mat Q$ and $\mat B$ are used below: the (non-symmetric) operator $\hat{\mat D} \coloneqq \mat Q - \mat B$, used in the strong form, and the skew-symmetric operator $\mat S \coloneqq 2\mat Q - \mat B = \mat Q - \mat Q^T$, used in the flux-differencing (split) form. 
The exact equivalence of the SBP-collocation identity to an SBP-SAT (simultaneous-approximation-term) finite-difference scheme with diagonal norm matrix was established by \citet{gassner2013skew}; we refer to \citet{kopriva2009implementing} for the standard textbook treatment of LGL-DGSEM and to \citet{kopriva2010quadrature} for the equivalence of the strong and weak forms used interchangeably below.

\paragraph{Standard (strong) form DGSEM.}
Approximating \eqref{eq:EulerPotentialTemperature} directly with the operators above yields, for element $\elemDom$,
\begin{align} \label{eq:DGSEM_3D_std}
	J_{ijk}\, \omega_{ijk}\, \dot{\state u}_{ijk}
	& + \omega_{jk}\left(\sum_{m=0}^{\polydeg} \hat D_{im}\, \contSt f^1_{mjk} - \delta_{i0}\numfluxb f^1_{(0,L)jk} + \delta_{i\polydeg}\numfluxb f^1_{(\polydeg,R)jk}\right) \nonumber \\
	& + \omega_{ik}\left(\sum_{m=0}^{\polydeg} \hat D_{jm}\, \contSt f^2_{imk} - \delta_{j0}\numfluxb f^2_{i(0,L)k} + \delta_{j\polydeg}\numfluxb f^2_{i(\polydeg,R)k}\right) \nonumber \\
	& + \omega_{ij}\left(\sum_{m=0}^{\polydeg} \hat D_{km}\, \contSt f^3_{ijm} - \delta_{k0}\numfluxb f^3_{ij(0,L)} + \delta_{k\polydeg}\numfluxb f^3_{ij(\polydeg,R)}\right)
	= J_{ijk}\,\omega_{ijk}\, \state s_{ijk},
\end{align}
where $J_{ijk}$ is the mapping Jacobian, $\omega_{ij}\coloneqq\omega_i\omega_j$ and $\omega_{ijk}\coloneqq\omega_i\omega_j\omega_k$, $\delta_{ij}$ is the Kronecker delta, $\state s$ is the source term (e.g., arising from the geopotential), and $L$/$R$ denote the neighboring element sharing the corresponding face. The numerical flux couples the two elements meeting at a face through an approximate Riemann solver $\numfluxb f(\cdot,\cdot\,;\vec n)$,
\begin{equation}
	\numfluxb f^1_{(\polydeg,R)jk} \coloneqq \hat{\blocktensor f}(\state u_{ \polydeg jk}, \state u_{Rjk}) \cdot (J\vec a^1)_{\polydeg jk}
	= \numfluxb f(\state u_{\polydeg jk}, \state u_{Rjk}; \vec n_{\polydeg jk})\, \norm{J\vec a^1}_{\polydeg jk},
\end{equation}
and analogously on the remaining five faces; the last equality uses the fact that the metric terms are single-valued across element interfaces, so the numerical flux reduces to the usual 1D Riemann problem in the normal direction, scaled by the surface Jacobian $\norm{J\vec a^1}_{\polydeg jk}$.

\paragraph{Flux-differencing (split) form DGSEM.}
In general, the standard form \eqref{eq:DGSEM_3D_std} does not satisfy a discrete entropy inequality for the nonlinear flux $\blocktensor f$.
\citet{fisher2013high} and \citet{gassner2016split} showed that this can be remedied by replacing the single-argument contravariant flux $\contSt f^m$ with a symmetric, consistent two-point volume flux $\blocktensor f^*(\state u_L,\state u_R)$, evaluated between every pair of nodes along each reference direction, and the operator $\hat{\mat D}$ with the skew-symmetric operator $\mat S$.
If the two-point flux is entropy conservative, the resulting volume terms do not produce entropy, and entropy stability is controlled by the numerical fluxes at the element interfaces.

For atmospheric flows, we can further discretize the gravity term $-\varrho\Nabla\phi$ as a non-conservative product (i.e., the product of the density, a component of the state, with the gradient of the geopotential) instead of as a source term.
This allows us to discretize the gravity term with the same SBP operator as the pressure gradient, so that the discrete pressure gradient and gravity term balance exactly for hydrostatic states, which is the key to obtaining a well-balanced scheme \cite{waruszewski2022entropy,artiano2025structure,artiano2026entropy}.
Non-conservative terms are incorporated into the flux-differencing framework through two-point non-conservative terms $\blocktensorG{\Phi}^{\star}(\state u_L,\state u_R)$, which are evaluated between the same pairs of nodes as the volume fluxes; this extension of entropy-stable split-form DGSEM schemes to non-conservative systems was developed by \citet{renac2019entropy} and \citet{bohm2020entropy}, and was subsequently used, e.g., in \cite{manzanero2020entropy,rueda2021entropy,ruedaramirez2024fluxdifferencing,waruszewski2022entropy,artiano2025structure}.
This yields the flux-differencing form
\begin{align} \label{eq:DGSEM_3D_fluxdiff}
	J_{ijk}\, \omega_{ijk}\, \dot{\state u}_{ijk}
	& + \omega_{jk}\left(\sum_{m=0}^{\polydeg} S_{im}\, \left (\contSt f^{1*}_{(i,m)jk} + \contSt \Phi^{1\star}_{(i,m)jk}\right )  - \delta_{i0}\numfluxb f^1_{(0,L)jk} + \delta_{i\polydeg}\numfluxb f^1_{(\polydeg,R)jk}\right) \nonumber \\
	& + \omega_{ik}\left(\sum_{m=0}^{\polydeg} S_{jm}\, \left ( \contSt f^{2*}_{i(j,m)k} + \contSt \Phi^{2\star}_{i(j,m)k} \right )- \delta_{j0}\numfluxb f^2_{i(0,L)k} + \delta_{j\polydeg}\numfluxb f^2_{i(\polydeg,R)k}\right) \nonumber \\
	& + \omega_{ij}\left(\sum_{m=0}^{\polydeg} S_{km}\, \left ( \contSt f^{3*}_{ij(k,m)} + \contSt \Phi^{3\star}_{ij(k,m)} \right ) - \delta_{k0}\numfluxb f^3_{ij(0,L)} + \delta_{k\polydeg}\numfluxb f^3_{ij(\polydeg,R)}\right)
	= J_{ijk}\,\omega_{ijk}\, \state s_{ijk},
\end{align}
with volume two-point conservative and nonconservative fluxes
\begin{equation}
	\contSt f^{1*}_{(i,m)jk} \coloneqq \blocktensor f^*(\state u_{ijk},\state u_{mjk}) \cdot \avg{J\vec a^1}_{(i,m)jk}, \quad 
    \contSt \Phi^{1\star}_{(i,m)jk} \coloneqq \contSt \Phi^{\star}(\state u_{ijk},\state u_{mjk}) \cdot \avg{J\vec a^1}_{(i,m)jk},
\end{equation}
and analogously for the $2$- and $3$-directions, where $\avg{\cdot}$ denotes the arithmetic average between the two nodal states involved. Notice that only the surface terms of \eqref{eq:DGSEM_3D_fluxdiff} are unchanged from \eqref{eq:DGSEM_3D_std}. 

We use the entropy-conserving two-point fluxes of \citet{artiano2025structure} for the volume terms in \eqref{eq:DGSEM_3D_fluxdiff}, which extend the general framework above to the potential-temperature formulation \eqref{eq:EulerPotentialTemperature}. 
In the split-form formulation, the geopotential is treated as a non-conservative product rather than a source term, guaranteeing additional robustness and well-balancedness for a given choice of the numerical flux of the density~\cite{waruszewski2022entropy, souza2023, artiano2025structure}.
Besides the arithmetic average $\avg{a} = \frac{1}{2} ( a_L + a_R)$ and the jump $\jump{a} = a_R - a_L$ of two nodal states, the two-point flux uses the logarithmic mean~\cite{ismail2009affordable}
\begin{equation}
    \avg{a}^{\mathrm{ln}} \coloneqq \frac{\jump{a}}{\jump{\ln a}}.
\end{equation}
The entropy-conserving two-point flux reads~\cite{artiano2025structure}
\begin{equation} \label{eq:flux_ec}
    \blocktensor f^{*}_{\ec}(\state u_L, \state u_R) =
    \begin{pmatrix}
        \avg{\varrho}^{\ln}\, \avg{\vec v} \\[2pt]
        \avg{\varrho}^{\ln}\, \avg{\vec v} \otimes \avg{\vec v} + \avg{p}\, \mat I \\[2pt]
        \avg{\varrho}^{\ln}\, \avg{\vec v} \,\big/\, \avg{1/\theta}^{\ln}
    \end{pmatrix},
\end{equation}
and the non-conservative two-point term is
\begin{equation} \label{eq:volume_gravity}
    \blocktensorG{\Phi}^{\star}(\state u_L, \state u_R)
    = \tfrac12\, \avg{\varrho}^{\ln}\, \jump{\phi}\,
      \big(0,\; \vec e_1,\; \vec e_2,\; \vec e_3,\; 0\big)^T.
\end{equation}

Note that the two-point non-conservative term in \eqref{eq:volume_gravity} is proportional to the jump of geopotential and, hence, does not contribute at the surface, since it is continuous across the interfaces and $\jump{\phi}$ therefore vanishes identically.

\paragraph{Surface flux.}
The numerical surface flux $\numfluxb f$ coupling two elements is computed with the entropy-stable and well-balanced two-point numerical flux of \citet{artiano2026entropy}, which is a modification of the Low Mach number Approximate Riemann Solver (LMARS) of \citet{chen2013control}.
With $\vec{n}$ the outward unit normal of Section~\ref{sec:dgsem} and the sound speed $c \coloneqq \sqrt{\gamma p/\varrho}$, the interface velocity and pressure of the entropy-stable flux of \citet{artiano2026entropy} are
\begin{equation}
    V^{\mathrm{num}} \coloneqq \avg{\vec{v} \cdot \vec{n}} - \frac{1}{2 \avg{c}} \left (\frac{\jump{p}}{\avg{\varrho}^{\ln}} + \jump{\phi} \right )  
\end{equation}
\begin{equation}
    P^{\mathrm{num}} \coloneqq \avg{p} - \frac{\avg{c} \avg{\varrho}}{2} \jump{\vec{v} \cdot \vec{n}}
\end{equation}
and the numerical flux reads
\begin{equation}
    \numfluxb f(\state u_L, \state u_R; \vec n) =
    \begin{pmatrix}
        \varrho^{\mathrm{up}}\, V^{\mathrm{num}} \\[4pt]
        (\varrho \vec v)^{\mathrm{up}}\, V^{\mathrm{num}} + P^{\mathrm{num}}\,\vec n \\[4pt]
        \varrho^{\mathrm{up}}\, V^{\mathrm{num}}
        \left(\dfrac{1}{\avg{1/\theta}^{\ln}}
              - \dfrac{\min(\theta_L,\theta_R)}{2\,\avg{\theta}}\, \jump{\theta}\, \mathrm{sign} \left ( V^{\mathrm{num}} \right ) \right)
    \end{pmatrix},
    \qquad
    a^{\mathrm{up}} \coloneqq \avg{a} - \tfrac12 \jump{a}\, \mathrm{sign}\left ( V^{\mathrm{num}} \right ).
\end{equation}

\subsection{The Rosenbrock-W method and the choice of splitting}
\label{sec:rosenbrock}

As discussed in Section~\ref{sec:intro}, we integrate the semi-discrete HEVI-DGSEM system in time with Rosenbrock-W methods \cite{steihaug1979attempt,norsett1979order,jahn2015asam}.
From now on, we will use uppercase arrays to denote the global vectors.
For instance, the vector $\state{U}$ denotes the global solution vector that assembles the individual nodal solutions of all nodes and elements in the domain.
Writing the system as $\dot{\state{U}} = \state{R}(\state{U}) = \state{R}_H(\state U) + \state{R}_V(\state U)$, with $\state R_H$ the horizontally explicit part and $\state R_V$ the vertically implicit part, each stage requires the solution of a single linear system involving an approximate Jacobian $\mat J$ of $\state{R}_V$, 
\begin{equation}\label{eq:rosenbrockW}
	\left( \frac{1}{\hat \gamma \Delta t} \mat{I} - \mat{J} \right) \state{K}_i = \state{R}\Big(\state{U}^n + \textstyle\sum_{j<i} \tilde{a}_{ij} \state{K}_j\Big) + \frac{1}{\Delta t}\textstyle\sum_{j<i} c_{ij} \state K_j,
	\qquad i = 1, \dots, s,
\end{equation}
and the stages are combined into the updated solution as
\begin{equation}\label{eq:Rosenbrock_update}
    \state U^{n+1} = \state U^n + \sum_{i=1}^{s} m_i \state K_i.
\end{equation}
Here, $s$ is the number of stages, and $\hat\gamma$, $\tilde{a
}_{ij}$, $c_{ij}$, and $m_i$ are the method coefficients.
They follow from the coefficients of the standard form of Rosenbrock-W methods \cite{hairer1996solving}, namely the strictly lower triangular matrix $\mat\alpha = (\alpha_{ij})$, the lower triangular matrix $\mat\Gamma = (\gamma_{ij})$ with constant diagonal $\gamma_{ii} = \hat\gamma$, and the weights $\state b = (b_i)$, as \cite[Sec.~IV.7]{hairer1996solving}
\begin{equation}\label{eq:rosenbrock_transformation}
	(\tilde{a}_{ij}) = \mat{ a}\,\mat\Gamma^{-1}, \qquad
	(c_{ij}) = \frac{1}{\hat\gamma}\mat I - \mat\Gamma^{-1}, \qquad
	(m_i)^T = \state b^T \mat\Gamma^{-1}.
\end{equation}
The standard form contains Jacobian--vector products, which are eliminated by the change of variables $\state K_i = \sum_{j\le i}\gamma_{ij}\tilde{\state K}_j$, where $\tilde{\state K}_i$ are the stage increments of the standard form.
As a consequence, $\mat J$ only enters \eqref{eq:rosenbrockW} through the system matrix: each stage requires one linear solve with $\tfrac{1}{\hat\gamma\Delta t}\mat I - \mat J$, but no Jacobian--vector products.
We use \eqref{eq:rosenbrockW}--\eqref{eq:Rosenbrock_update} in the implementation, and $\mat a$, $\mat\Gamma$, and $\state b$ to derive the method, since the order conditions and the stability function take their simplest form in terms of these coefficients.
Since the HEVI-DGSEM system considered here is autonomous, we have omitted the additional term that accounts for an explicit time dependence of the right-hand side \cite{hairer1996solving}.
All tableaus used in this work have the same constant diagonal coefficient $\hat \gamma$ for every stage, so that the factorization of $\tfrac{1}{\hat \gamma\Delta t}\mat I - \mat J$ can be reused across all stages of a time step instead of being recomputed stage by stage.

The defining property of W-methods is that their order conditions do not restrict the matrix $\mat J$: they are satisfied for an arbitrary approximation of the Jacobian \cite{steihaug1979attempt,norsett1979order}.
The choice of $\mat J$ therefore determines the stability and efficiency of the scheme and, for stiff problems, its observed accuracy, but not the order conditions.
This is what makes the nested Schur-complement solvers of Section~\ref{sec:nestedschur} possible: we can replace the exact, dense Jacobian of the split-form DGSEM operator by a cheaper, sparser approximation, provided that it captures the stiff, vertically propagating acoustic and gravity-wave modes that motivate the HEVI splitting.

Table~\ref{tab:ssp_tableau} contains the tableau of the Rosenbrock-W method used throughout this paper.
This method is derived in Appendix~\ref{app:tableau} and fulfills three desired properties: (i) it has a constant diagonal $\hat \gamma$ for all stages, (ii) it reduces to the three-stage, third-order SSP Runge--Kutta method of \citet{shu1988efficient} in purely explicit settings, and (iii) it is L-stable.

\begin{table}[h]
    \centering
	\caption{Coefficients of the three-stage, second-order, L-stable Rosenbrock-W method with $\hat\gamma = 1 - \tfrac{1}{\sqrt2}$ used in this work.
	Left: row sums $a_i$ (first column), $a_{ij}$ and $b_i$ (middle block), and $\gamma_{ij}$ (right block); for $\mat J = \mat 0$, the method reduces to the third-order SSP Runge--Kutta method of \citet{shu1988efficient}.
	Right: coefficients of \eqref{eq:rosenbrockW}--\eqref{eq:Rosenbrock_update}, obtained with \eqref{eq:rosenbrock_transformation}.}
    \label{tab:ssp_tableau}
    \[
    \renewcommand{\arraystretch}{1.3}
    \begin{array}{c|ccc|ccc}
        0             &            &            &            & 1-\tfrac1{\sqrt2}    &                   & \\
        1             & 1          &            &            & \tfrac{3}{\sqrt2}-3  & 1-\tfrac1{\sqrt2} & \\
        \tfrac12      & \tfrac14   & \tfrac14   &            & \tfrac{3\sqrt2-4}{8} & -\tfrac14         & 1-\tfrac1{\sqrt2} \\
        \hline
                      & \tfrac16   & \tfrac16   & \tfrac23   &                      &                   &
    \end{array}
    \qquad
    \begin{array}{c cc cc c}
        i & \tilde{a}_{i1} & \tilde{a}_{i2} & c_{i1} & c_{i2} & m_i \\
        \hline
        1 & & & & & \tfrac{26+15\sqrt2}{6} \\
        2 & 2+\sqrt2 & & -3(2+\sqrt2) & & \tfrac{8+5\sqrt2}{6} \\
        3 & 2+\sqrt2 & \tfrac{2+\sqrt2}{4} & -\tfrac{18+11\sqrt2}{4} & -\tfrac{3+2\sqrt2}{2} & \tfrac{2(2+\sqrt2)}{3}
    \end{array}
    \]
\end{table}

\paragraph{Choice of the approximate Jacobian.}
The matrix $\mat J$ in \eqref{eq:rosenbrockW} is our approximation to the Jacobian of the \emph{full} right-hand side $\state R = \state R_H + \state R_V$.
Exploiting the freedom of W-methods, we make an aggressive choice: we set the horizontal contribution to zero and retain only the Jacobian of the vertical operator, $\mat J \approx \partial \state R_V/\partial \state U$, evaluated with the standard (non-split) form of the DGSEM discretization, rather than the entropy-conservative split form that is used to evaluate $\state R$ in \eqref{eq:rosenbrockW}.
This choice has two consequences that are central to the rest of the paper.
First, $\mat J$ is block diagonal, with one block per vertical column, so that the linear systems of all columns are independent, which also simplifies the MPI parallelization.
Second, within each column, the Jacobian has a sparsity pattern that can be exploited by the nested Schur-complement solvers of Section~\ref{sec:nestedschur}; its derivation is the subject of Section~\ref{sec:jacobian_hevi}.
Including the split-form gravity term in $\mat J$, as opposed to the standard source term used here, would introduce additional fill-in: the density block coupling to the vertical momentum would no longer be diagonal.
We therefore intentionally use an inexact, non-split-form Jacobian throughout.
Although this means that there are some inconsistencies between the split-form DGSEM discretization of the equations and the standard DGSEM treatment of the vertical implicit systems, we remark that the discretization of the implicit systems can be interpreted as a preconditioning technique to deal with the stiffness of the problem.
Furthermore, we will show that the choice of approximate Jacobian and time-integration method preserves linear invariants for our system, which results in total mass and potential temperature conservation up to round-off errors.

\subsection{Definition of the HEVI Jacobian}
\label{sec:jacobian_hevi}
The HEVI splitting treats only the third reference direction implicitly.
To obtain the vertical subsystem, we project the momentum equation of \eqref{eq:EulerPotentialTemperature} onto the unit vector normal to the coordinate surfaces $\xi^3 = \mathrm{const}$,
\begin{equation}
	\vec\nu \coloneqq \frac{J\vec a^3}{n}, \qquad n \coloneqq \norm{J\vec a^3},
\end{equation}
where $J\vec a^3$ are the discrete metric terms of Section~\ref{sec:dgsem}, and denote the resulting vertical momentum by $\varrho\tilde w \coloneqq \vec\nu\cdot\varrho\vec v$.
Note that, due to the polynomial representation of the metric terms and in the presence of orography, $\vec\nu$ is in general not aligned with the radial direction, and that it may vary along a column.

In our HEVI scheme, ``vertical'' means the direction of the third contravariant vector (i.e., the direction of $J\vec{a}^3$ and $\vec\nu$), not the radial direction. That is the correct HEVI split for terrain-following meshes as $\vec\nu$ points in the direction where the stiffness actually is: normal to the $\xi^3 = \text{const}$ surfaces, which is the direction of the small grid spacing (see Figure~\ref{fig:shell_piece}). 

\begin{figure}
  \centering
  \begin{subfigure}[t]{0.48\linewidth}
    \centering
\includegraphics[width=\linewidth]{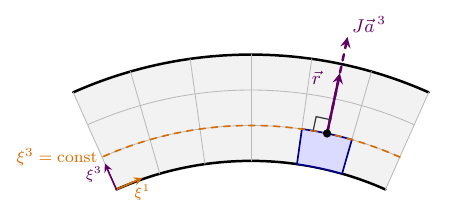}
    \caption{Sphere without orography}
    \label{fig:shell_piece_a}
  \end{subfigure}\hfill
  \begin{subfigure}[t]{0.48\linewidth}
    \centering
\includegraphics[width=\linewidth]{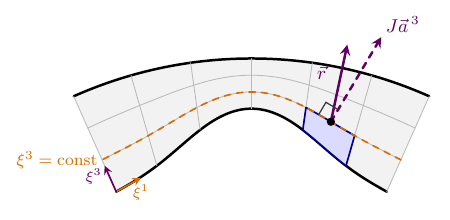}
    \caption{Non-constant orography}
    \label{fig:shell_piece_b}
  \end{subfigure}
  \caption{Illustration of the contravariant metric term orientation and radial direction for a sphere without orography (a) and a sphere with non-constant orography (b).}
  \label{fig:shell_piece}
\end{figure}

Denoting by $\mat T = \varrho\vec v\otimes\vec v + p\,\mat I$ the momentum flux tensor, the projection of its vertical flux divergence reads
\begin{equation}\label{eq:projection_identity}
	\vec\nu\cdot\partial_{\xi^3}\big(\mat T\,J\vec a^3\big)
	= \partial_{\xi^3}\big(\vec\nu\cdot\mat T\,J\vec a^3\big) - \big(\partial_{\xi^3}\vec\nu\big)\cdot\mat T\,J\vec a^3,
\end{equation}
where the last term corresponds to the Christoffel-symbol terms of the Euler equations in covariant form; see, e.g., \cite{baldauf2021horizontally,baldauf2020discontinuous,baldauf2026comparison} for covariant formulations in DG methods for the compressible Euler equations, and \cite{Montoya2026} for an entropy-stable DGSEM in covariant form.

Since $J\vec a^3\cdot\varrho\vec v = n\,\varrho\tilde w$, the vertical fluxes of mass and potential temperature are also expressed exactly in terms of $\varrho\tilde w$.
We discretize the vertical subsystem using the standard (strong) form DGSEM detailed in Section \ref{sec:dgsem}, retaining only the third direction of~\eqref{eq:DGSEM_3D_std}:
\begin{equation}
    J_{k} \omega_k \dot{\state{u}}^V_{k} + \sum_{m =0}^{\polydeg} \hat{D}_{km} \contSt f^{3V}_{m} - \delta_{k0}\numfluxb f^{3V}_{(0,L)} + \delta_{k\polydeg} \numfluxb f^{3V}_{(\polydeg, R)} = J_{k} \omega_k \state s^V_{k}, \quad k = 0, \dots, \polydeg,
\end{equation}
where the horizontal indices $ij$ have been dropped for readability, and $L$ and $R$ now denote the elements below and above.
The vector of prognostic quantities of the vertical subsystem is $\state{u}^V = (\varrho, \varrho \tilde w, \varrho \theta)^T$, the contravariant flux is $\contSt f^{3V} = n\,(\varrho \tilde w,\; \varrho \tilde w^2 + p,\; \varrho \theta \tilde w)$, and the source term $\state s^V_k = (0, -\tilde g_k\,\varrho_k + s_{\mathrm{geo}}, 0)^T$ contains the gravity term $\tilde g_k \coloneqq \vec\nu_k\cdot\Nabla\phi(\vec x_k)$ and geometric source terms $s_{\mathrm{geo}}$ that arise from the second term on the right-hand side of \eqref{eq:projection_identity}.
\begin{remark}[Unit-vector formulation vs.\ covariant formulation]
    The choice of the unit vector $\vec\nu$, rather than the contravariant vector $J\vec a^3$ itself, avoids geometric source terms in the acoustic subsystem.
Since $\vec\nu$ has unit length, $\big(\partial_{\xi^3}\vec\nu\big)\cdot\vec\nu = 0$, and the pressure contribution to the last term vanishes identically, i.e., $\vec\nu\cdot\partial_{\xi^3}\big(p\,J\vec a^3\big) = \partial_{\xi^3}(n\,p)$, for any mesh geometry.
The remaining geometric term, $-\big(\partial_{\xi^3}\vec\nu\big)\cdot\varrho\vec v\,(J\vec a^3\cdot\vec v)$, is advective and is neglected below together with the bulk advection.
Contributions from the horizontal derivatives, including the curvature terms on the sphere and the cross terms of non-orthogonal meshes, belong to the horizontal operator $\state R_H$.
In contrast, projecting onto $J\vec a^3$ instead of $\vec\nu$ would introduce the additional pressure term $-\tfrac12\,p\,\partial_{\xi^3}\|{J\vec a^3}\|^2$ in the vertical subsystem.
\end{remark}

The fast (stiff) modes that we are trying to capture with our Rosenbrock-W methods are the vertically propagating sound waves.
For the HEVI Jacobian, it is therefore sufficient to linearize only the acoustic part of the vertical flux.
To this end, we freeze the potential temperature $\theta_f = (\varrho\theta)/\varrho$ at its nodal values of the state $\state u^n$ at which the Jacobian is assembled, once per time step, so that it is not differentiated, and we neglect the bulk advection term $\varrho\tilde w^2$ in the momentum flux.
This yields the restricted flux function $\contSt f^{3V,c} = n\,(\varrho \tilde w,\; p,\; \theta_f \varrho \tilde w)$, in which the pressure $p = p(\varrho\theta)$ remains the only nonlinearity.
Note that this reduces to the acoustic-gravity subsystem with the approximate flux Jacobian
\begin{equation}\label{eq:fluxJac_app}
    \bigpartialderiv{\contSt f^{3V}}{\state{u}^V} \approx
    \bigpartialderiv{\contSt f^{3V,c}}{\state{u}^{V}} =
    n
    \begin{pmatrix}
        0 & 1 & 0\\
        0 & 0 & \frac{\partial p}{\partial (\varrho \theta)}\\
        0 & \theta_f & 0
    \end{pmatrix}
    =
    n
    \begin{pmatrix}
        0 & 1 & 0\\
        0 & 0 & \frac{\gamma p}{\varrho \theta}\\
        0 & \theta_f & 0
    \end{pmatrix},
\end{equation}
with the eigenvalues
\begin{equation}
    \lambda_1=0, \qquad
    \lambda_{2,3} = \pm n\sqrt{\frac{\gamma p}{\varrho \theta} \theta_f} = \pm n\sqrt{\gamma \frac{p}{\varrho}} = \pm n\, c,
\end{equation}
which shows that the approximate flux Jacobian \eqref{eq:fluxJac_app} captures the upward- and downward-moving acoustic waves; the metric factor $n$ converts the physical sound speed $c$ into the corresponding wave speed in reference space.
Note that $\theta_f$ matches $\theta$ point-wise when the Jacobian is constructed.

To define the vertical Jacobians, we use the following ordering for the vertical solution vector that gathers all degrees of freedom of a vertical subsystem,
\begin{equation}\label{eq:orderingJac}
    \state{U}^V = (\dots, \varrho_0^{(e)}, \dots, \varrho_{\polydeg}^{(e)}, \varrho \theta_0^{(e)}, \dots, \varrho \theta_{\polydeg}^{(e)},\varrho \tilde w_0^{(e)}, \dots, \varrho \tilde w_{\polydeg}^{(e)}, \dots),
\end{equation}
where the super-index $(e)$ denotes the element index.
The semi-discrete vertical subsystem of element $e$ reads, node by node ($k = 0,\dots,\polydeg$),
\begin{equation}\label{eq:vertical_residual}
\begin{aligned}
    \dot\varrho^{(e)}_k &= -\frac{1}{J^{(e)}_k\omega_k}\left[\sum_{m=0}^{\polydeg} \hat D_{km}\, n^{(e)}_m (\varrho\tilde w)^{(e)}_m
        + \delta_{k\polydeg}\, \numflux{f}^{\varrho}_{e+\frac12} - \delta_{k0}\, \numflux{f}^{\varrho}_{e-\frac12}\right] \eqqcolon R^{\varrho,(e)}_k,\\
    \dot{(\varrho\theta)}^{(e)}_k &= -\frac{1}{J^{(e)}_k\omega_k}\left[\sum_{m=0}^{\polydeg} \hat D_{km}\, n^{(e)}_m \theta_{f,m}^{(e)} (\varrho\tilde w)^{(e)}_m
        + \delta_{k\polydeg}\, \numflux{f}^{\varrho \theta}_{e+\frac12} - \delta_{k0}\, \numflux{f}^{\varrho \theta}_{e-\frac12}\right] \eqqcolon R^{\varrho\theta,(e)}_k,\\
    \dot{\left(\varrho\tilde{w}\right)}^{(e)}_k &= -\frac{1}{J^{(e)}_k\omega_k}\left[\sum_{m=0}^{\polydeg} \hat D_{km}\, n^{(e)}_m p^{(e)}_m
        + \delta_{k\polydeg}\, \numflux{f}^{\varrho \tilde w}_{e+\frac12} - \delta_{k0}\, \numflux{f}^{\varrho \tilde w}_{e-\frac12}\right]
        - \tilde g^{(e)}_k\, \varrho^{(e)}_k \eqqcolon R^{\varrho\tilde w,(e)}_k,
\end{aligned}
\end{equation}
where $e\pm\frac12$ denote the upper and lower interfaces of element $e$.
At the element interfaces, i.e., the coordinate surfaces $\xi^3 = -1$ and $\xi^3 = 1$, the metric terms are single-valued and $\vec\nu$ coincides with the unit normal vector, so $\varrho\tilde w$ is the normal momentum and $n$ is the surface Jacobian.
Consistently with the restricted flux $\contSt f^{3V,c}$, the numerical fluxes retain only the acoustic part of the LMARS-type flux of Section~\ref{sec:dgsem}~\cite{chen2013control}, multiplied by the surface Jacobian.
At the interface $e+1/2$, with left state at node $\polydeg$ of element $e$ and right state at node $0$ of element $e+1$, they read
\begin{equation}\label{eq:vertical_numflux}
\begin{aligned}
    \numflux{f}^{\varrho}_{e+\frac12} &= n_{e+\frac12}\left(\avg{\varrho\tilde w} - \frac{1}{2c_0}\jump{p}\right),\\
    \numflux{f}^{\varrho\theta}_{e+\frac12} &= n_{e+\frac12}\left(\avg{\theta_f\,\varrho\tilde w} - \frac{\avg{\theta_f}}{2c_0}\jump{p}\right),\\
    \numflux{f}^{\varrho\tilde w}_{e+\frac12} &= n_{e+\frac12}\left(\avg{p} - \frac{c_0}{2}\jump{\varrho\tilde w}\right),
\end{aligned}
\end{equation}
where $n_{e+\frac12}$ is the (single-valued) value of $n$ at the interface and $c_0$ is a constant reference sound speed fixed at the beginning of the simulation (as is common to do in atmospheric applications), e.g., $c_0 = 340\,\mathrm{m/s}$.
This differs from the entropy-stable and well-balanced LMARS flux of Section~\ref{sec:dgsem}, which uses the local sound speed in the explicit right-hand side; since W-methods only require a consistent approximation of the Jacobian, this simplification affects only the stability and efficiency of the scheme, not its formal accuracy.
Note that the constant reference sound speed can be replaced with a frozen local sound speed at time $n$ in such a way that it does not change the sparsity.
To maintain conservation properties, the dissipative term of the $\varrho\theta$ flux is evaluated using the average of the nodally frozen potential temperature $\theta_f$.

The central part of \eqref{eq:vertical_numflux} has a useful interpretation.
Since $\sum_{m} S_{km}\,\avg{\contSt f}_{(k,m)} = \sum_m \hat D_{km}\,\contSt f_m$ for the arithmetic-mean two-point flux, \eqref{eq:vertical_residual} coincides with the flux-differencing form \eqref{eq:DGSEM_3D_fluxdiff} with central volume fluxes and LMARS-type surface dissipation.
Hence, the Jacobian derived below may equivalently be regarded as the Jacobian of a split-form DGSEM with central two-point fluxes, as long as no metric dealiasing is performed.

\paragraph{Linearization.}
All quantities that are frozen in $\contSt f^{3V,c}$, namely $\theta_f$ and $\avg{\theta}$, are evaluated at the state where $\mat J$ is assembled and are not differentiated.
The only remaining nonlinearity is the equation of state, which we linearize through
\begin{equation}
    \alpha_k \coloneqq \frac{\partial p}{\partial(\varrho\theta)}\bigg|_k = \frac{\gamma\, p_k}{(\varrho\theta)_k}.
\end{equation}
To simplify the notation, we drop the subscript $f$ from $\theta$ in the remainder of this section and introduce the element-local nodal vectors $\stateG \alpha^{(e)}$, $\stateG\theta^{(e)}$, and $\state n^{(e)}$, e.g., $\stateG \alpha^{(e)} = (\alpha^{(e)}_0,\dots,\alpha^{(e)}_{\polydeg})^T$, as well as the interface averages
$$
\avg{\theta}_{e+\frac12} \coloneqq \frac12\big(\theta^{(e)}_{\polydeg} + \theta^{(e+1)}_0\big).
$$

We also introduce the diagonal matrices
\begin{equation}
\begin{gathered}
    \mat \Lambda(\stateG a) \coloneqq \diag(a_0,\dots,a_{\polydeg}), \qquad
    \mat{\mathcal M}^{(e)} \coloneqq \diag\big(J^{(e)}_0\omega_0,\dots,J^{(e)}_{\polydeg}\omega_{\polydeg}\big),\\
    \mat B_{n}^{(e)} \coloneqq \diag\big(n_{e-\frac12},0,\dots,0,n_{e+\frac12}\big), \qquad
    \mat B_{n\theta}^{(e)} \coloneqq \diag\big(n_{e-\frac12}\avg{\theta}_{e-\frac12},0,\dots,0,n_{e+\frac12}\avg{\theta}_{e+\frac12}\big),
\end{gathered}
\end{equation}
where $\mat\Lambda(\stateG a)$ is defined for any state vector $\stateG a$.
The matrices $\state n^{(e)}$ and $\mat B_n^{(e)}$ depend only on the metric terms and can be precomputed once; the state enters the Jacobian only through $\stateG\alpha^{(e)}$, $\stateG\theta^{(e)}$, and $\avg{\theta}_{e\pm\frac12}$.
Differentiating \eqref{eq:vertical_residual} and collecting the volume and surface contributions of each element, we repeatedly encounter the combination
\begin{equation}
    \hat{\mat D} + \tfrac12\mat B = \mat Q - \tfrac12 \mat B = \tfrac12\left(\mat Q - \mat Q^T\right) = \tfrac12 \mat S,
\end{equation}
i.e., the central part of the linearized operator is exactly the skew-symmetric SBP operator of the split form.

\paragraph{Block structure.}
With the ordering of $\state U^V$, the Jacobian of a vertical column $\mat J^V = \partial \state R^V/\partial \state U^V$ is block tridiagonal, with $3\times3$ blocks of size $(\polydeg+1)\times(\polydeg+1)$ coupling each element to itself and to its lower and upper neighbors,
\begin{equation}
    \mat J^V =
    \begin{pmatrix}
        \ddots & \ddots & \ddots & & \\
        & \mat J^{(e,e-1)} & \mat J^{(e,e)} & \mat J^{(e,e+1)} & \\
        & & \ddots & \ddots & \ddots
    \end{pmatrix}.
\end{equation}
The element-diagonal block, with rows and columns ordered as $(\varrho, \varrho\theta, \varrho\tilde w)$, is
\begin{equation}\label{eq:Jac_ee}
    \mat J^{(e,e)} = -\frac12 \big(\mat{\mathcal M}^{(e)}\big)^{-1}
    \begin{bmatrix}
        \mat 0 & \frac{1}{c_0}\mat B_{n}^{(e)}\,\mat\Lambda(\stateG\alpha^{(e)}) & \mat S\,\mat\Lambda(\state n^{(e)}) \\[4pt]
        \mat 0 & \frac{1}{c_0}\mat B_{n\theta}^{(e)}\,\mat\Lambda(\stateG\alpha^{(e)}) & \mat S\,\mat\Lambda(\state n^{(e)} \circ \stateG\theta^{(e)}) \\[4pt]
        2\,\mat{\mathcal M}^{(e)}\,\mat\Lambda(\tilde{\state g}^{(e)}) & \mat S\,\mat\Lambda(\state n^{(e)} \circ \stateG\alpha^{(e)}) & c_0\,\mat B_{n}^{(e)}
    \end{bmatrix},
\end{equation}
where we use Hadamard products, e.g., $\state n^{(e)} \circ \stateG\alpha^{(e)}$, for the entry-wise multiplication of nodal vectors.
The coupling blocks only connect the two nodes adjacent to an interface:
\begin{equation}\label{eq:Jac_em1}
    \mat J^{(e,e-1)} = \frac{n_{e-\frac12}}{2J^{(e)}_0\omega_0}
    \begin{bmatrix}
        0 & \dfrac{\alpha^{(e-1)}_{\polydeg}}{c_0} & 1 \\[8pt]
        0 & \dfrac{\avg{\theta}_{e-\frac12}\,\alpha^{(e-1)}_{\polydeg}}{c_0} & \theta^{(e-1)}_{\polydeg} \\[8pt]
        0 & \alpha^{(e-1)}_{\polydeg} & c_0
    \end{bmatrix} \otimes \state e_0\state e_{\polydeg}^T,
\end{equation}
\begin{equation}\label{eq:Jac_ep1}
    \mat J^{(e,e+1)} = -\frac{n_{e+\frac12}}{2J^{(e)}_{\polydeg}\omega_{\polydeg}}
    \begin{bmatrix}
        0 & -\dfrac{\alpha^{(e+1)}_0}{c_0} & 1 \\[8pt]
        0 & -\dfrac{\avg{\theta}_{e+\frac12}\,\alpha^{(e+1)}_0}{c_0} & \theta^{(e+1)}_0 \\[8pt]
        0 & \alpha^{(e+1)}_0 & -c_0
    \end{bmatrix} \otimes \state e_{\polydeg}\state e_0^T.
\end{equation}
Here, $\state e_0, \state e_{\polydeg} \in \mathbb R^{\polydeg+1}$ are the first and last canonical unit vectors.
In the first and last elements of a column, the missing neighbor block is dropped.
The linearization of the boundary flux, which depends only on the interior trace, then contributes to the corresponding corner entries of $\mat J^{(e,e)}$.

Three features of \eqref{eq:Jac_ee}--\eqref{eq:Jac_ep1} motivate the chosen ordering and are exploited in Section~\ref{sec:nestedschur}.
First, the density enters the linearized system only through the gravity term in the $\varrho\tilde w$ equation, so the first block column vanishes except for the diagonal block $-\mat\Lambda(\tilde{\state g}^{(e)})$.
Second, the $\varrho$ and $\varrho\theta$ equations couple to $\varrho$ and $\varrho\theta$ only through diagonal entries and through the pair of nodes adjacent to each interface, which come from the surface dissipation.
Hence the leading $2\times 2$ block of $\tfrac{1}{\hat\gamma\Delta t}\mat I - \mat J^V$ is trivial to invert.
Third, the dense coupling through the LGL operator $\mat S$ appears only in the $\varrho\tilde w$ column and in the $(\varrho\tilde w,\varrho\theta)$ block.
Eliminating $\varrho$ and $\varrho\theta$ therefore leads to a Schur complement acting on the vertical momentum alone, as we show next.

Figure \ref{fig:Jperm_global} shows the sparsity pattern of the Jacobian matrix derived above after including the gravity and time-integration terms.
Figure \ref{fig:Jperm_only_vars} uses the ordering \eqref{eq:orderingJac}, while Figure~\ref{fig:Jperm} uses the ``natural'' ordering, in which the variables are interleaved node by node.
Note that the sparsity of the flux Jacobian \eqref{eq:fluxJac_app} is evident in the latter ordering.
\begin{figure}
    \centering
    \begin{subfigure}[b]{0.48\textwidth}
        \centering
        \includegraphics[width=\textwidth]{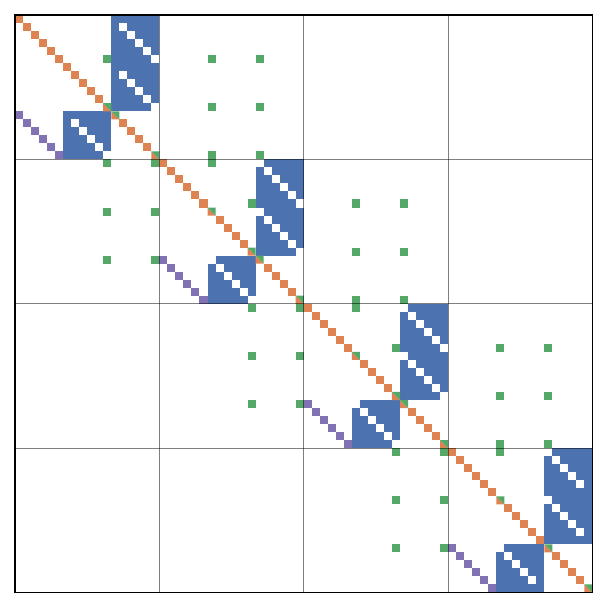}
        \caption{Ordering $\state U^V$ (grouped by variable within each element).}
        \label{fig:Jperm_only_vars}
    \end{subfigure}
    \hfill
    \begin{subfigure}[b]{0.48\textwidth}
        \centering
        \includegraphics[width=\textwidth]{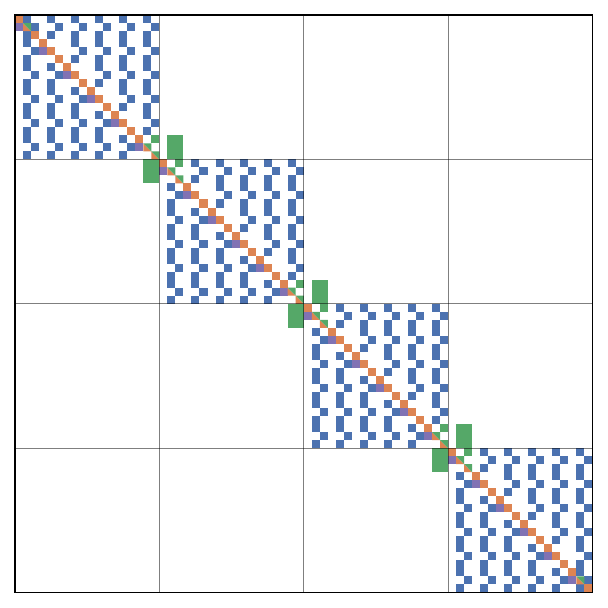}
        \caption{Natural ordering ($\varrho$, $\varrho\tilde w$, $\varrho\theta$ interleaved node by node).}
        \label{fig:Jperm}
    \end{subfigure}

    \vspace{0.5em}
    \includegraphics[width=0.9\textwidth]{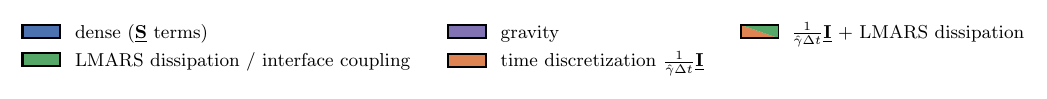}

    \caption{Sparsity pattern of $\tfrac{1}{\hat\gamma\Delta t}\mat I - \mat J^V$ for one vertical column with $4$ elements and $\polydeg = 5$.}
    \label{fig:Jperm_global}
\end{figure}

\paragraph{Projection onto the vertical subsystem.}
Since the vertical subsystem is formulated in the variables $\state u^V = (\varrho, \varrho\theta, \varrho\tilde w)^T$, whereas the global discretization evolves $\state u = (\varrho, \varrho\vec v, \varrho\theta)^T$, the stage equations \eqref{eq:rosenbrockW} cannot be solved with $\mat J^V$ directly.
To connect both sets of variables, we define at every node the restriction and lift operators
\begin{equation}\label{eq:restriction_lift}
    \mat P_k \coloneqq
    \begin{pmatrix}
        1 & \vec 0^T & 0 \\
        0 & \vec 0^T & 1 \\
        0 & \vec\nu_k^T & 0
    \end{pmatrix} \in \mathbb R^{3\times 5},
    \qquad
    \mat L_k \coloneqq
    \begin{pmatrix}
        1 & 0 & 0 \\
        \vec 0 & \vec 0 & \vec\nu_k \\
        0 & 1 & 0
    \end{pmatrix} \in \mathbb R^{5\times 3},
\end{equation}
which satisfy $\mat P_k \mat L_k = \mat I$, since $\vec\nu_k$ has unit length.
The restriction $\mat P_k$ performs the same projection of the momentum onto $\vec\nu_k$ that we applied to the governing equations, while the lift $\mat L_k$ maps an increment of the vertical momentum back to an increment of the Cartesian momentum along $\vec\nu_k$.
We collect the nodal operators into the block-diagonal matrices $\mat P$ and $\mat L$, which also reorder the degrees of freedom column by column as in $\state U^V$.
The approximate Jacobian of the full right-hand side in \eqref{eq:rosenbrockW} is then
\begin{equation}
    \mat J \coloneqq \mat L\, \mat J^V \mat P,
\end{equation}
which vanishes on all horizontal couplings and on the momentum components orthogonal to $\vec\nu$, and is therefore an admissible choice for a W-method.

Denoting by $\state B_i$ the right-hand side of \eqref{eq:rosenbrockW} for stage $i$ and applying $\mat P$ to the stage equation, the property $\mat P\mat L = \mat I$ yields the reduced column-wise systems
\begin{equation}\label{eq:vertical_stage}
    \left( \frac{1}{\hat \gamma \Delta t} \mat{I} - \mat{J}^V \right) \state{K}^V_i = \mat P\, \state B_i,
    \qquad \state K^V_i \coloneqq \mat P\, \state K_i.
\end{equation}
Once $\state K^V_i$ is known, the stage equation gives $\state K_i = \hat \gamma\Delta t\,(\state B_i + \mat L\,\mat J^V \state K^V_i)$, and eliminating $\mat J^V \state K^V_i$ with \eqref{eq:vertical_stage} leads to
\begin{equation}\label{eq:vertical_lift}
    \state K_i = \hat \gamma\Delta t\, \state B_i + \mat L \left( \state K^V_i - \hat \gamma\Delta t\, \mat P\,\state B_i \right).
\end{equation}
Equation \eqref{eq:vertical_lift} has a simple interpretation.
The density and potential-temperature components of $\state K_i$ coincide with those of $\state K^V_i$.
For the momentum, the component of the increment along $\vec\nu_k$ is replaced by the implicit vertical solution, while the components orthogonal to it are treated explicitly, i.e.,
\begin{equation}
    \state K^{\varrho\vec v}_{i,k} = \hat \gamma\Delta t\,\big(\mat I - \vec\nu_k\vec\nu_k^T\big)\,\state B^{\varrho\vec v}_{i,k} + \vec\nu_k\, K^{\varrho\tilde w}_{i,k}.
\end{equation}
In summary, every time a linear solve is required in \eqref{eq:rosenbrockW}, we project the right-hand side onto the vertical subsystem with $\mat P$, solve the independent column systems \eqref{eq:vertical_stage}, and lift the result back with \eqref{eq:vertical_lift}.
The matrices $\mat P$ and $\mat L$ are never formed explicitly.
Both act node by node: the restriction copies the density and the potential-temperature density and computes the vertical momentum as the scalar product $\vec\nu_k\cdot(\varrho\vec v)_k$, and the lift adds the increment of the vertical momentum to the momentum along $\vec\nu_k$.
Hence, both operations only require the nodal versors $\vec\nu_k$, and their cost is negligible compared with that of the column solves.
Both projections act node by node and are negligible in cost compared with the column solves.
As in any Rosenbrock-type method, $\mat J^V$ is assembled only once per time step, at $\state U^n$, and re-used in all stages.
Since W-methods only require an approximate Jacobian, this matrix can moreover neglect the horizontal couplings, use frozen coefficients without reducing the order of the scheme, and even be frozen for more than one time step.
The latter advantage is not explored in this work.
Since $\hat \gamma$ is the same for all stages, the same holds for the operator $\tfrac{1}{\hat \gamma\Delta t}\mat I - \mat J^V$ and for its nested Schur-complement factorization described in Section~\ref{sec:nestedschur}. 

\begin{remark}[Other formulations of the Euler equations]
\label{rem:other_formulations}
As noted in Section~\ref{sec:equations}, the pressure depends only on the potential-temperature density $\varrho\theta$ in the formulation \eqref{eq:EulerPotentialTemperature0}.
Consequently, the density does not appear in the approximate flux Jacobian \eqref{eq:fluxJac_app} at all, and it enters the vertical subsystem only through the pointwise gravity term.
This is not the case for other formulations of the compressible Euler equations: if the thermodynamic equation is written, e.g., for the total energy density, the pressure depends on all prognostic variables, and the corresponding flux Jacobian is considerably denser.
In Section~\ref{sec:nestedschur}, we exploit this sparsity to eliminate density and potential-temperature density analytically through local Schur complements.

The methods presented here can nonetheless be applied to dynamical cores based on other formulations.
Since Rosenbrock-W methods admit any approximation of the Jacobian, the vertical subsystems can always be formulated in terms of $(\varrho, \varrho\tilde w, \varrho\theta)$.
To this end, the state at which the Jacobian is assembled is converted to the potential-temperature formulation, and, in every stage, the right-hand side $\state B_i$ and the stage vector $\state K_i$ are mapped between both sets of variables with the linearization of this change of variables.
These transformations act node by node and can be fused into the projections \eqref{eq:restriction_lift}, so their cost is small compared with the column solves.

Using the potential-temperature formulation in both the explicit and the implicit parts is nonetheless beneficial for three reasons: (i) no transformations are required; (ii) since the approximate Jacobian is derived from the same conservative discretization, the implicit stages conserve $\varrho\theta$ to machine precision, whereas the node-wise transformations generally break the discrete conservation of the thermodynamic variable of the dynamical core in the implicit stages (a proof of Rosenbrock-W's conservation of linear invariants is provided in Appendix~\ref{app:linear_invariants}); and (iii) structure-preserving discretizations are available for this formulation \cite{artiano2025structure}.
\end{remark}

\section{Nested Schur-complement approach}
\label{sec:nestedschur}

As identified in \cite{rueda2021statically,laskowski2020advantages,pan2025half}, LGL-DGSEM discretizations produce sparsity and connectivity patterns that admit efficient implicit solvers based on Schur-complement reductions of the permuted Jacobian matrix.
In the CG and DG literature, these approaches, first introduced for continuous finite elements in \cite{guyan1965reduction}, are commonly referred to as \textit{static condensation} methods.
In such methods, element-interior degrees of freedom are eliminated locally, reducing the size of the globally coupled linear system by a factor that depends on the polynomial degree and spatial dimensionality of the discretization.
This technique is closely related to hybridized discontinuous Galerkin (HDG) methods \cite{cockburn2009unified,nguyen2009implicit,nguyen2009nonlinear,nguyen2011implicit}.
This reduction is particularly effective for vertically one-dimensional formulations, such as those arising in HEVI methods.

On a different line, \citet{giraldo2010semi} derived different Schur formulations for semi-implicit atmospheric models and studied how the choice of prognostic variables affects the structure of the resulting linear systems.
These formulations apply Schur-complement reductions to the linearized algebraic systems in order to eliminate selected prognostic variables and obtain smaller coupled problems.
More recently, these techniques have been combined successfully with DG discretizations in \cite{reddy2023schur}, where significant speedups were demonstrated by reducing the effective system size to approximately one fifth of the original three-dimensional problem.

In this section, we combine these two strategies into a \emph{nested} Schur-complement method, applied independently to the system matrix $\tfrac{1}{\hat\gamma\Delta t}\mat I - \mat J^V$ of each column of $N_z$ vertically stacked elements, i.e., to the column systems \eqref{eq:vertical_stage}.
Before describing the two levels at which this nesting occurs, we briefly recall the underlying algebraic tool.
Given a linear system partitioned into two groups of unknowns $\state x$ and $\state y$,
\begin{equation}\label{eq:schurgeneric}
	\begin{pmatrix} \mat{\mathcal A} & \mat{\mathcal B} \\ \mat{\mathcal C} & \mat{\mathcal D} \end{pmatrix}
	\begin{pmatrix} \state x \\ \state y \end{pmatrix}
	=
	\begin{pmatrix} \state f \\ \state g \end{pmatrix},
\end{equation}
solving the first block row for $\state x = \mat{\mathcal A}^{-1}(\state f - \mat{\mathcal B} \state y)$ (assuming $\mat{\mathcal A}$ is invertible) and substituting into the second row eliminates $\state x$ entirely, leaving the smaller system
\begin{equation}\label{eq:schurgeneric2}
	\left(\mat{\mathcal D} - \mat{\mathcal C} \mat{\mathcal A}^{-1} \mat{\mathcal B}\right) \state y = \state g - \mat{\mathcal C} \mat{\mathcal A}^{-1} \state f
\end{equation}
for $\state y$ alone.
The matrix $\mat{\mathcal S} = \mat{\mathcal D} - \mat{\mathcal C}\mat{\mathcal A}^{-1}\mat{\mathcal B}$ is the \emph{Schur complement} of $\mat{\mathcal A}$ in the full system \eqref{eq:schurgeneric}; once $\state y$ is obtained from the reduced system \eqref{eq:schurgeneric2}, $\state x$ is recovered by back-substitution.
Eliminating $\state x$ this way is only worthwhile when $\mat{\mathcal A}$ is itself cheap to invert -- e.g., block-diagonal, or otherwise structured -- so that forming $\mat{\mathcal A}^{-1}\mat{\mathcal B}$ and $\mat{\mathcal A}^{-1}\state f$ is inexpensive.

We apply exactly this mechanism twice, at two different scales:
\begin{itemize}
	\item A \emph{local} Schur complement, performed once per element, that exploits the prognostic-variable block structure to analytically eliminate two of the three implicit variables from each elemental block $\mat{\mathcal A}_e$ (Section~\ref{sec:localschur}). This local elimination is analogous to the strategies in \cite{giraldo2010semi, reddy2023schur}.
	\item A \emph{global} Schur complement, performed once per column, that eliminates the (now cheaply invertible) element-interior unknowns in favor of a small set of elemental edge unknowns (Section~\ref{sec:globalschur}). This global elimination is inspired by the methods of \cite{rueda2021statically,laskowski2020advantages,pan2025half}.
\end{itemize}
The two condensation strategies compared in this paper, \emph{two-sided} and \emph{one-sided}, share the same local elimination, but assign degrees of freedom (DOFs) to the local block $\mat{\mathcal A}$ versus the globally coupled block $\mat{\mathcal D}$ differently; each is described as a self-contained strategy in Sections~\ref{sec:twosided} and~\ref{sec:onesided}.

\subsection{Local Schur complement: ordering and elimination of prognostic variables}
\label{sec:localschur}

In the remainder of this section, we number the LGL nodes of each element as $1,\dots,M$, with $M = \polydeg + 1$, instead of using the indexing $0,\dots,\polydeg$ as in Section~\ref{sec:hevidgsem}.
In particular, nodes $1$ and $M$ are the lower and upper edge nodes of the element.
The vertical right-hand side and its Jacobian are commonly assembled following the layout of the DGSEM solution array, which stores the three implicitly treated variables interleaved node by node: for an element with $M$ LGL nodes, the common ordering lists $(\varrho, \varrho\tilde w, \varrho\theta)$ at node 1, then at node 2, and so on.
Figure~\ref{fig:var_perm_natural} shows the resulting sparsity pattern of the diagonal block of the system matrix for one element.
In this ordering, the coupling structure is scattered throughout the matrix rather than organized into clean blocks.

Grouping the DOFs by variable instead, i.e., using the ordering $\state U^V$ of Section~\ref{sec:jacobian_hevi} (first all $\varrho$ DOFs of the element, then all $\varrho\theta$, then all $\varrho\tilde w$), is itself a permutation, and it brings the elemental block $\mat{\mathcal A}_e \coloneqq \tfrac{1}{\hat\gamma\Delta t}\mat I - \mat J^{(e,e)}$ into the block form
\begin{equation}\label{eq:Ablockstructure}
	\mat{\mathcal A}_e =
	\begin{pmatrix}
		\mat{\mathcal A}_{11} & \mat{\mathcal A}_{12} & \mat{\mathcal A}_{13} \\
		\mat{\mathcal A}_{21} & \mat{\mathcal A}_{22} & \mat{\mathcal A}_{23} \\
		\mat{\mathcal A}_{31} & \mat{\mathcal A}_{32} & \mat{\mathcal A}_{33}
	\end{pmatrix},
\end{equation}
shown in Figure~\ref{fig:var_perm_grouped}, where block rows/columns $1,2,3$ correspond to $\varrho, \varrho\theta, \varrho\tilde w$ and each block is of size $M\times M$.
From \eqref{eq:Jac_ee}, the blocks read
\begin{equation}\label{eq:Ablocks}
\begin{aligned}
	&\mat{\mathcal A}_{11} = \tfrac{1}{\hat\gamma\Delta t}\mat I,
	&&\mat{\mathcal A}_{12} = \tfrac{1}{2c_0} \big(\mat{\mathcal M}^{(e)}\big)^{-1} \mat B_{n}^{(e)}\,\mat\Lambda(\stateG\alpha^{(e)}),
	&&\mat{\mathcal A}_{13} = \tfrac12 \big(\mat{\mathcal M}^{(e)}\big)^{-1} \mat S\,\mat\Lambda(\state n^{(e)}),\\
	&\mat{\mathcal A}_{21} = \mat 0,
	&&\mat{\mathcal A}_{22} = \tfrac{1}{\hat\gamma\Delta t}\mat I + \tfrac{1}{2c_0} \big(\mat{\mathcal M}^{(e)}\big)^{-1} \mat B_{n\theta}^{(e)}\,\mat\Lambda(\stateG\alpha^{(e)}),
	&&\mat{\mathcal A}_{23} = \tfrac12 \big(\mat{\mathcal M}^{(e)}\big)^{-1} \mat S\,\mat\Lambda(\state n^{(e)}\stateG\theta^{(e)}),\\
	&\mat{\mathcal A}_{31} = \mat\Lambda(\tilde{\state g}^{(e)}),
	&&\mat{\mathcal A}_{32} = \tfrac12 \big(\mat{\mathcal M}^{(e)}\big)^{-1} \mat S\,\mat\Lambda(\state n^{(e)}\stateG\alpha^{(e)}),
	&&\mat{\mathcal A}_{33} = \tfrac{1}{\hat\gamma\Delta t}\mat I + \tfrac{c_0}{2} \big(\mat{\mathcal M}^{(e)}\big)^{-1} \mat B_{n}^{(e)}.
\end{aligned}
\end{equation}
The blocks $\mat{\mathcal A}_{11}$, $\mat{\mathcal A}_{22}$, and $\mat{\mathcal A}_{33}$ are diagonal: the density and potential-temperature equations do not depend on the density, so $\mat{\mathcal A}_{11}$ only contains the time-discretization term, while $\mat{\mathcal A}_{22}$ and $\mat{\mathcal A}_{33}$ add to it the LMARS dissipation, which only acts on the two edge nodes of the element (Remark~\ref{rem:lmars}).
For the same reason, $\mat{\mathcal A}_{12}$ has only two nonzero entries, one at each edge node.
The gravity block $\mat{\mathcal A}_{31}$ is diagonal because we use the standard, pointwise gravity source term instead of the split gravity term.
The remaining blocks $\mat{\mathcal A}_{13}$, $\mat{\mathcal A}_{23}$, and $\mat{\mathcal A}_{32}$ are dense, inherited from the skew-symmetric SBP operator $\mat S$; hence, their diagonal vanishes, except at the physical bottom and top boundaries, where the wall flux contributes to the corner entries.

\begin{remark}[Effect of the LMARS surface flux]
\label{rem:lmars}
The blocks $\mat{\mathcal A}_{12}$, $\mat{\mathcal A}_{22}$, and $\mat{\mathcal A}_{33}$ contain the linearized LMARS dissipation of the numerical fluxes \eqref{eq:vertical_numflux} \cite{chen2013control}, which only acts at the two edge nodes of each element.
In particular, the dissipation of the density flux is proportional to the jump of $\tilde p$, which depends on $\varrho\theta$; this is the origin of the coupling between $\varrho$ and $\varrho\theta$ in $\mat{\mathcal A}_{12}$, which, for an interior element, is zero everywhere except at its two corner entries $(1,1)$ and $(M,M)$.
At the two physical (bottom and top) boundaries, the reflecting-wall flux has no pressure jump, so the corresponding corner entries of $\mat{\mathcal A}_{12}$ and $\mat{\mathcal A}_{22}$ vanish, whereas the corner entry of $\mat{\mathcal A}_{33}$ remains.
\end{remark}

\begin{figure}
	\centering
	\begin{subfigure}[b]{0.4\linewidth}
		\centering
		\includegraphics[width=\linewidth]{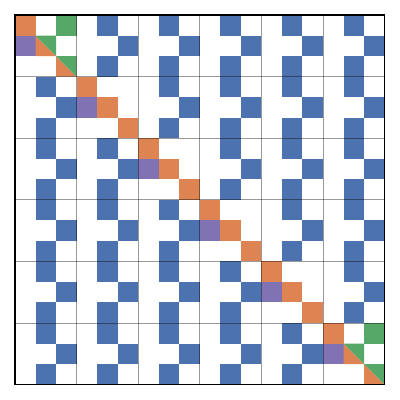}
		\caption{Natural, node-major ordering.}
		\label{fig:var_perm_natural}
	\end{subfigure}
	\hspace{0.05\linewidth}
	\begin{subfigure}[b]{0.4\linewidth}
		\centering
		\includegraphics[width=\linewidth]{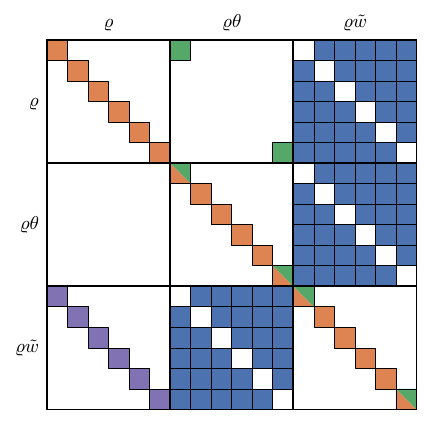}
		\caption{Grouped by variable.}
		\label{fig:var_perm_grouped}
	\end{subfigure}

	\vspace{0.5em}
	\includegraphics[width=0.9\linewidth]{figs/legend_sparsity.pdf}

	\caption{Sparsity pattern of the diagonal block of $\tfrac{1}{\hat\gamma\Delta t}\mat I - \mat J^V$ for one generic (interior) element.
	\subref{fig:var_perm_natural} Natural, node-major ordering: $\varrho, \varrho\tilde w, \varrho\theta$ interleaved node by node; the block structure of Eq.~\eqref{eq:Ablockstructure} is present, but not visible.
	\subref{fig:var_perm_grouped} After grouping the DOFs by variable ($\varrho$, then $\varrho\theta$, then $\varrho\tilde w$), the same nonzero pattern rearranges into the block form of Eq.~\eqref{eq:Ablockstructure} (the two isolated corner entries in the $(\varrho,\varrho\theta)$ block are the LMARS coupling of Remark~\ref{rem:lmars}).
	}
	\label{fig:var_perm}
\end{figure}

It is precisely this structure -- two ``light'' variables ($\varrho$, $\varrho\theta$) that are (almost) diagonally self-coupled, and one ``heavy'' variable ($\varrho\tilde w$) that carries all of the dense LGL coupling -- that enables the local Schur complement of $\mat{\mathcal A}_e$ to be taken analytically, once per element, independently of anything happening at the element edges.

At the level of the whole column, Figure~\ref{fig:Jperm_global} in Section~\ref{sec:jacobian_hevi} shows the sparsity pattern of the system matrix in two orderings: the natural, node-major ordering of the solution array (Fig.~\ref{fig:Jperm}) and the ordering $\state U^V$, which groups the DOFs of each element by variable (Fig.~\ref{fig:Jperm_only_vars}).
At each interface, the three prognostic variables of one side's edge node couple to $\varrho\theta$ and $\varrho\tilde w$ of the neighbor's edge node through the LMARS common interface state, cf.\ the coupling blocks \eqref{eq:Jac_em1} and \eqref{eq:Jac_ep1}.
In the natural ordering, these couplings form contiguous off-diagonal blocks of size $3\times2$ (Fig.~\ref{fig:Jperm}), whereas in the grouped ordering the same entries are scattered across the element blocks (Fig.~\ref{fig:Jperm_only_vars}).
In both orderings, the column matrix is block tridiagonal, with one block of size $3M\times3M$ per element and off-diagonal blocks that only couple neighboring elements.
This structure can be exploited by solving the system with the block Thomas algorithm, i.e., a block LU factorization that sweeps once through the column, factorizing one $3M\times3M$ block per element.
Such an approach treats the element blocks as dense; the nested Schur-complement approach developed in this section, in contrast, exploits their internal structure.

\subsection{Global Schur complement: interior/edge elimination}
\label{sec:globalschur}

Having established that each elemental block $\mat{\mathcal A}_e$ is cheap to invert (Section~\ref{sec:localschur}), the second, global level of the method applies the same Schur-complement mechanism of Eqs.~\eqref{eq:schurgeneric}--\eqref{eq:schurgeneric2} once more, at the scale of the whole column.
Reordering the system matrix of the column so that all DOFs assigned to the local block precede all DOFs assigned to the globally coupled block yields the $2\times2$ partition
\begin{equation}\label{eq:ABCD}
	\frac{1}{\hat\gamma\Delta t}\mat I - \mat J^V
	\;\longrightarrow\;
	\begin{pmatrix}
		\mat{\mathcal A} & \mat{\mathcal B} \\
		\mat{\mathcal C} & \mat{\mathcal D}
	\end{pmatrix},
\end{equation}
where $\mat{\mathcal A}$ is block-diagonal with one cheaply invertible block $\mat{\mathcal A}_e$ per element, and $\mat{\mathcal B}$ and $\mat{\mathcal C}$ couple the $\mat{\mathcal A}$-DOFs of each element only to its own $\mat{\mathcal D}$-DOFs and to those of its neighboring elements.
Since part of the element DOFs are moved to $\mat{\mathcal D}$, the block $\mat{\mathcal A}_e$ in \eqref{eq:ABCD} is the restriction of \eqref{eq:Ablockstructure} to the DOFs retained in $\mat{\mathcal A}$; it keeps the $3\times3$ block structure of \eqref{eq:Ablockstructure}, with correspondingly smaller blocks.
Because $\mat{\mathcal A}$ is block-diagonal, its inverse is computed independently and in parallel, one element at a time; the global Schur complement $\mat{\mathcal S} = \mat{\mathcal D} - \mat{\mathcal C} \mat{\mathcal A}^{-1} \mat{\mathcal B}$ is then a small system of size $\mathcal O(N_z)$ rather than $\mathcal O(N_z M)$.

In this paper, we devise two strategies to permute the system matrix of Fig.~\ref{fig:Jperm_only_vars}, such that the block structure of Eq.~\eqref{eq:ABCD} is obtained with a block-diagonal matrix $\mat{\mathcal A}$.
These two strategies will be termed \emph{two-sided interface condensation} and \emph{one-sided interface condensation}.
Exactly {which} DOFs are assigned to $\mat{\mathcal A}$ versus $\mat{\mathcal D}$ differs between the two strategies: they draw this line differently, per prognostic variable and per element side.
Figure~\ref{fig:two_strategies_assignment} summarizes this assignment node by node for one element, and Figs.~\ref{fig:two_strategies_twosided} and~\ref{fig:two_strategies_onesided} mark the rows and columns of the column matrix that each strategy moves to $\mat{\mathcal D}$.
In both cases, the same assignment is applied to every element, including the first and last elements of the column, so that the data structures are identical for all elements (see Remark~\ref{rem:topboundary}).
We now describe each strategy in full.

\begin{figure}
	\centering
	\begin{subfigure}[b]{0.28\linewidth}
		\centering
		\includegraphics[width=\linewidth]{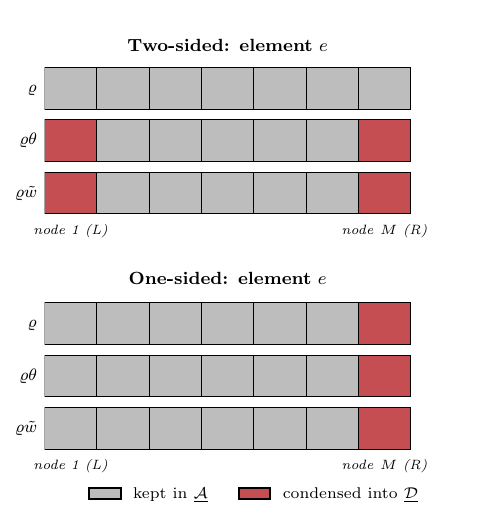}
		\caption{DOF assignment per element.}
		\label{fig:two_strategies_assignment}
	\end{subfigure}
	\hfill
	\begin{subfigure}[b]{0.34\linewidth}
		\centering
		\includegraphics[width=\linewidth]{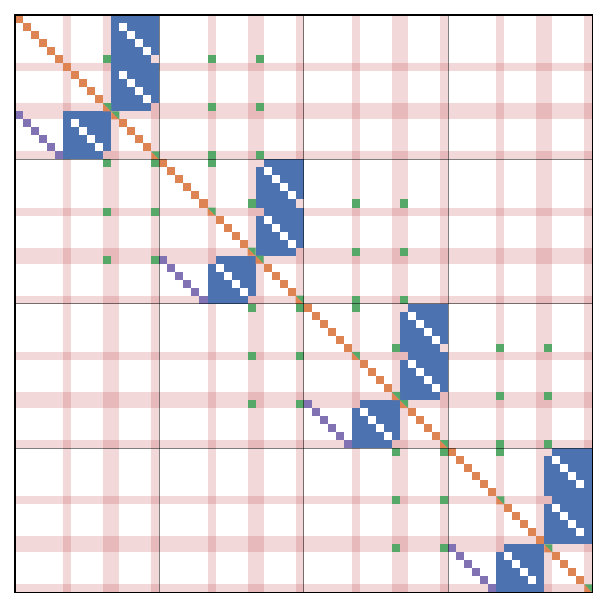}
		\caption{Two-sided condensation.}
		\label{fig:two_strategies_twosided}
	\end{subfigure}
	\hfill
	\begin{subfigure}[b]{0.34\linewidth}
		\centering
		\includegraphics[width=\linewidth]{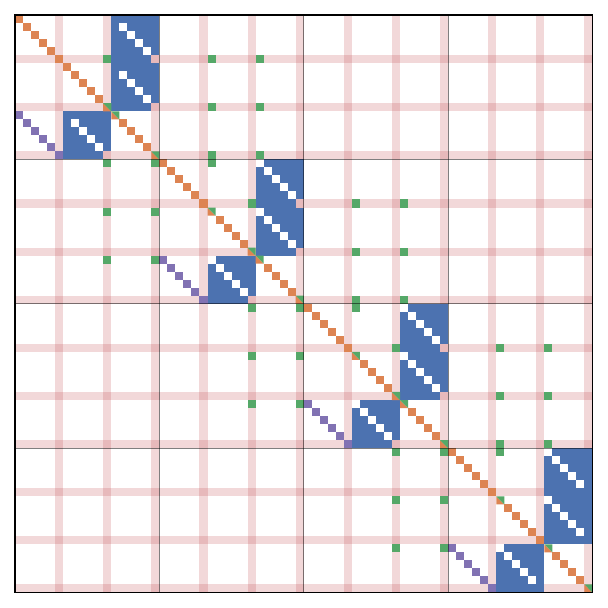}
		\caption{One-sided condensation.}
		\label{fig:two_strategies_onesided}
	\end{subfigure}

	\vspace{0.5em}
	\includegraphics[width=0.9\linewidth]{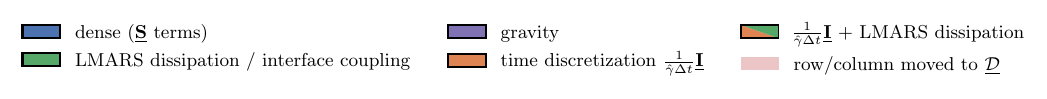}

	\caption{Assignment of the DOFs to the local block $\mat{\mathcal A}$ versus the globally coupled block $\mat{\mathcal D}$.
	\subref{fig:two_strategies_assignment} Node-wise assignment for one element: two-sided condensation never condenses $\varrho$, and condenses $\varrho\theta$ and $\varrho\tilde w$ at \emph{both} element edges; one-sided condensation condenses all three variables, but only at a single edge (node $M$) per element, leaving node $1$ -- a genuine interface DOF -- inside $\mat{\mathcal A}$.
	\subref{fig:two_strategies_twosided}, \subref{fig:two_strategies_onesided} System matrix of Fig.~\ref{fig:Jperm_only_vars} ($N_z=4$, $M=6$) with the rows and columns moved to $\mat{\mathcal D}$ marked in red for each strategy.
	}
	\label{fig:two_strategies}
\end{figure}

\subsubsection{Two-sided interface condensation}
\label{sec:twosided}

In the two-sided strategy, the density is \emph{never} condensed: all $\varrho$ unknowns of an element, including those at its two edge nodes, are assigned to $\mat{\mathcal A}$.
Only the potential-temperature and vertical-momentum unknowns at the two edge nodes (node $1$ and node $M$) of each element are assigned to $\mat{\mathcal D}$ (Fig.~\ref{fig:two_strategies_twosided}), giving $\mat{\mathcal D}$ four rows/columns per element, ordered as $\big((\varrho\theta)_1, (\varrho\tilde w)_1, (\varrho\tilde w)_M, (\varrho\theta)_M\big)$, and a total size of $4N_z\times 4N_z$.
The permuted system matrix of the column is shown in Figure~\ref{fig:twosided_full}.

\begin{figure}
	\centering
	\includegraphics[width=0.6\linewidth]{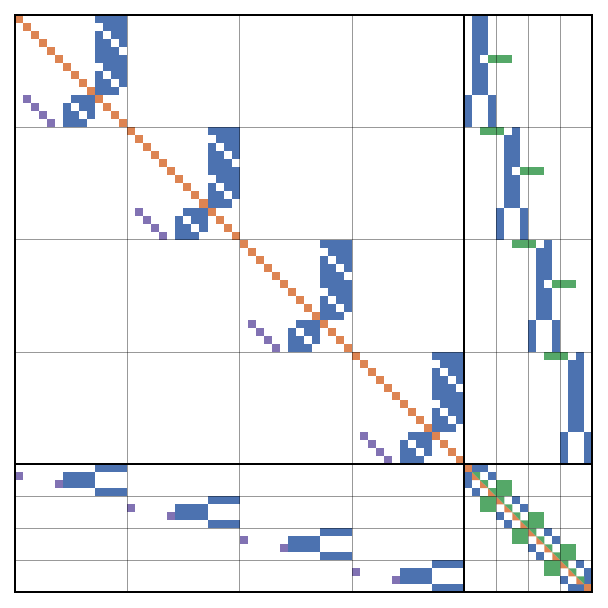}
	\caption{Sparsity pattern of the column system matrix after the two-sided permutation ($N_z=4$, $M=6$); the thick lines separate the blocks $\mat{\mathcal A}$, $\mat{\mathcal B}$, $\mat{\mathcal C}$, and $\mat{\mathcal D}$ of Eq.~\eqref{eq:ABCD}.
	The local block $\mat{\mathcal A}$ (top left) is block-diagonal, with one block per element in which density, potential temperature, and momentum are grouped separately.
	Colors as in Fig.~\ref{fig:var_perm}.}
	\label{fig:twosided_full}
\end{figure}

A direct consequence is that the local block $\mat{\mathcal A}_e$ is no longer uniform across variables: $\varrho$ retains all $M$ nodes, while $\varrho\theta$ and $\varrho\tilde w$ retain only the $M-2$ interior nodes (see Fig.~\ref{fig:AS_comparison}, top left).
The blocks of Eq.~\eqref{eq:Ablockstructure} survive with correspondingly rectangular shapes; e.g., the dense block $\mat{\mathcal A}_{13}$ is now of size $M\times(M-2)$ rather than square.
Moreover, since both corner entries of $\mat{\mathcal A}_{12}$ belong to the condensed $\varrho\theta$ DOFs, we obtain $\mat{\mathcal A}_{12} = \mat 0$.
As a result, the Schur complement of each elemental block reads
\begin{equation}\label{eq:SA}
	\mat{\mathcal S}_{e} = \mat{\mathcal A}_{33} - \mat{\mathcal A}_{31}\mat{\mathcal A}_{11}^{-1}\mat{\mathcal A}_{13} - \mat{\mathcal A}_{32}\mat{\mathcal A}_{22}^{-1}\mat{\mathcal A}_{23},
\end{equation}
where we recall that $\mat{\mathcal A}_{11}$ and $\mat{\mathcal A}_{22}$ are diagonal and, hence, trivially invertible.

As can be seen in Fig.~\ref{fig:AS_comparison} (middle left), $\mat{\mathcal D}$ is itself block-tridiagonal in the element index before any elimination has taken place, since the edge unknowns of neighboring elements couple directly through the numerical flux.
The correction $-\mat{\mathcal C}\mat{\mathcal A}^{-1}\mat{\mathcal B}$ fills the diagonal blocks of the Schur complement $\mat{\mathcal S}$, without changing the block-tridiagonal footprint (Fig.~\ref{fig:AS_comparison}, bottom left).
Hence, $\mat{\mathcal S}$ has dense $4\times 4$ diagonal blocks.
The sub- and super-diagonal blocks are also $4\times 4$, but two rows and two columns of each off-diagonal block are identically zero, leaving a single dense $2\times 2$ subblock as the only coupling between adjacent elements: rows 3--4 and columns 1--2 in each super-diagonal block, and rows 1--2 and columns 3--4 in each sub-diagonal block.
As a banded matrix, $\mat{\mathcal S}$ therefore has lower bandwidth $k_l = 3$ and upper bandwidth $k_u = 3$, for a total of seven nonzero bands, and it is factorized by a self-written banded LU decomposition without pivoting at a cost of approximately $2 \cdot 4N_z \cdot k_l \cdot k_u = 72\,N_z$ floating-point operations per column, where we count a multiply--add as two operations.
Forward--backward substitution costs approximately $2 \cdot 4N_z \cdot (k_l + k_u) = 48\,N_z$ operations per right-hand side.

\subsubsection{One-sided interface condensation}
\label{sec:onesided}

In the one-sided strategy, all three variables ($\varrho, \varrho\theta, \varrho\tilde w$) are condensed, but only at \emph{one} edge per element: the unknowns at node $M$ (the upper edge) are assigned to $\mat{\mathcal D}$, whereas node $1$ (the lower edge, a genuine interface DOF coupled to the element below) is kept in $\mat{\mathcal A}$, together with all interior nodes (Fig.~\ref{fig:two_strategies_onesided}).
This gives $\mat{\mathcal D}$ three rows/columns per element, ordered as $\big(\varrho_M, (\varrho\tilde w)_M, (\varrho\theta)_M\big)$ for consistency with our implementation, and a total size of $3N_z\times 3N_z$, smaller than the $4N_z\times 4N_z$ of the two-sided strategy.
The permuted system matrix of the column is shown in Figure~\ref{fig:onesided_full}.

\begin{figure}
	\centering
	\includegraphics[width=0.6\linewidth]{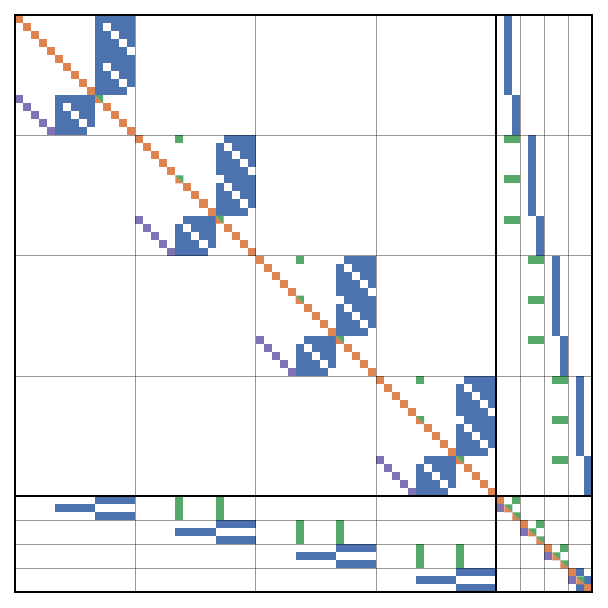}
	\caption{Sparsity pattern of the column system matrix after the one-sided permutation ($N_z=4$, $M=6$); the thick lines separate the blocks $\mat{\mathcal A}$, $\mat{\mathcal B}$, $\mat{\mathcal C}$, and $\mat{\mathcal D}$ of Eq.~\eqref{eq:ABCD}.
	Compared to Fig.~\ref{fig:twosided_full}, the local block $\mat{\mathcal A}$ absorbs one extra edge node (node $1$) per element and is uniform across variables.
	Colors as in Fig.~\ref{fig:var_perm}.}
	\label{fig:onesided_full}
\end{figure}

A direct consequence is that the local block $\mat{\mathcal A}_e$ is uniform across variables: $\varrho$, $\varrho\theta$, and $\varrho\tilde w$ retain the same $M-1$ nodes, so $\mat{\mathcal A}_e$ has exactly the structure of Eq.~\eqref{eq:Ablockstructure}, with blocks that are one node smaller (see Fig.~\ref{fig:AS_comparison}, top right).
This is a simpler, more symmetric situation than the non-uniform $\mat{\mathcal A}_e$ of the two-sided strategy.
Since node $1$ is retained in $\mat{\mathcal A}_e$, the block $\mat{\mathcal A}_{12}$ keeps the corner entry of Remark~\ref{rem:lmars} at node $1$ (except in the first element, where it vanishes), while the node-$M$ corner entry moves, together with the rest of node $M$, into $\mat{\mathcal D}$.
An expression similar to \eqref{eq:SA} is used to obtain the Schur complement of the elemental block, with a rank-one correction due to the single nonzero entry of $\mat{\mathcal A}_{12}$.

A further consequence of assigning node $1$ to $\mat{\mathcal A}$ rather than $\mat{\mathcal D}$ is that the raw block $\mat{\mathcal D}$ of Eq.~\eqref{eq:ABCD} has \emph{no} off-diagonal blocks at all: the interface partners of the node-$M$ unknowns of element $e$ are the node-$1$ unknowns of element $e+1$, and since the latter belong to $\mat{\mathcal A}$, the corresponding entries fall inside $\mat{\mathcal B}$ and $\mat{\mathcal C}$, which couple $\mat{\mathcal D}$ to the neighbor's \emph{local} block, and not inside $\mat{\mathcal D}$ itself.
$\mat{\mathcal D}$ is therefore exactly block-diagonal (Fig.~\ref{fig:AS_comparison}, middle right).
The block-tridiagonal structure of the Schur complement $\mat{\mathcal S}$ (Fig.~\ref{fig:AS_comparison}, bottom right) is produced \emph{entirely} by the correction $-\mat{\mathcal C}\mat{\mathcal A}^{-1}\mat{\mathcal B}$ of Eq.~\eqref{eq:schurgeneric2}.
This is a structurally different mechanism from the two-sided strategy, where the tridiagonal coupling is already present, directly, in the interface flux.

The diagonal blocks of $\mat{\mathcal S}$ are dense $3\times 3$ matrices with one structural zero: the $(3,1)$ entry, which couples the density edge unknown to the potential-temperature edge equation, vanishes identically for all elements.
Similarly, the sub- and super-diagonal blocks have an identically zero first column, associated with the density edge unknown, and are therefore stored as $3\times 2$ matrices (\textit{LU-blocks}).
Both zeros have the same origin: the density enters the momentum equation only through the pointwise gravity term, so the column of $\mat{\mathcal B}$ associated with $\varrho_M$ vanishes, and the density column of $\mat{\mathcal S}$ coincides with that of the block-diagonal $\mat{\mathcal D}$.
The $(1,3)$ entry of each diagonal block, by contrast, is generally nonzero; in $\mat{\mathcal D}$, it is exactly the node-$M$ LMARS corner entry of Remark~\ref{rem:lmars}, which lives entirely inside $\mat{\mathcal D}$ since node $M$ is condensed in every element.

As a conventional banded matrix, $\mat{\mathcal S}$ has lower bandwidth $k_l = 4$ and upper bandwidth $k_u = 5$, for a total of ten nonzero bands, so that a generic banded LU decomposition would cost approximately $2 \cdot 3N_z \cdot k_l \cdot k_u = 120\,N_z$ floating-point operations per column.
Exploiting the block-tridiagonal structure is more efficient: we apply a block Thomas (tridiagonal) algorithm in which, owing to the zero first column of the LU-blocks and the zero $(3,1)$ entry, the density unknown of each block can be eliminated with scalar operations, and the active Schur complement at each elimination step reduces to a dense $2\times 2$ system that can be inverted analytically.
Each step of the block factorization requires the scalar elimination of the density unknown ($\approx 14$ operations), two products of $2\times2$ matrices ($\approx 28$ operations), and one analytical $2\times 2$ inversion ($\approx 8$ operations), for a total of approximately $50N_z$ floating-point operations per column.
Forward--backward substitution with the factored system costs approximately $36N_z$ operations per right-hand side, compared with $2 \cdot 3N_z \cdot (k_l + k_u) = 54\,N_z$ for a generic banded solve.

\begin{figure}
	\centering
	\includegraphics[width=0.72\linewidth]{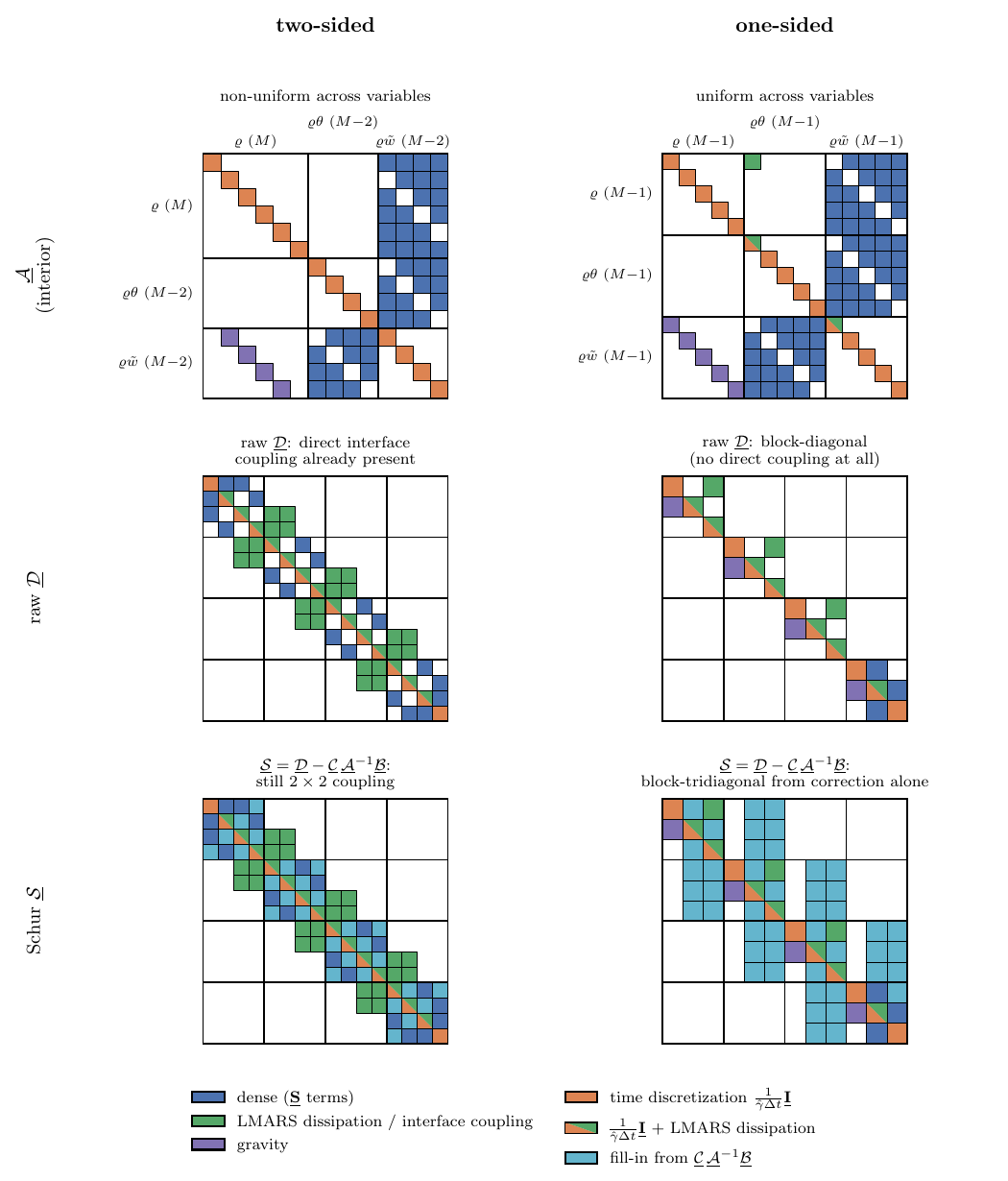}
	\caption{Side-by-side comparison of the local block $\mat{\mathcal A}_e$ of one interior element (top row), the raw block $\mat{\mathcal D}$ before elimination (middle row), and the global Schur complement $\mat{\mathcal S}$ after elimination (bottom row), for the two-sided (left column) and one-sided (right column) strategies, with $N_z=4$ and $M=6$.
	Entries of $\mat{\mathcal S}$ that are zero in $\mat{\mathcal D}$ and created by the correction $-\mat{\mathcal C}\mat{\mathcal A}^{-1}\mat{\mathcal B}$ are marked as fill-in.
	The contrast between the middle and bottom rows is the key structural difference between the two strategies: for one-sided condensation, all of the inter-element coupling in $\mat{\mathcal S}$ is induced by the correction term, since the raw $\mat{\mathcal D}$ has none.}
	\label{fig:AS_comparison}
\end{figure}

\begin{remark}[Treatment of the top boundary]
\label{rem:topboundary}
Both condensation strategies condense the upper edge of the top element (node $M$ of element $N_z$) exactly like every interior interface, even though that node lies on the physical top boundary and therefore never couples to a neighbor above.
It would be possible to leave those DOFs in $\mat{\mathcal A}$ instead, at the cost of a larger, non-uniform block for the last element alone.
We use the uniform version for both strategies, since it keeps every elemental block the same size and shape, which simplifies the implementation at the cost of a fixed, small number of extra condensed DOFs per column ($3$ for one-sided, $2$ for two-sided) that does not grow with $N_z$.
\end{remark}

\section{Implementation}
\label{sec:implementation}

The nested Schur-complement HEVI solvers are built on top of \href{https://github.com/trixi-framework/TrixiAtmo.jl}{TrixiAtmo.jl} \cite{trixiatmo}, the atmospheric extension of the high-order DG framework \href{https://github.com/trixi-framework/Trixi.jl}{Trixi.jl} \cite{schlottkelakemper2025trixi, ranocha2022adaptive,schlottkelakemper2021purely,candelaresi2026massivelyparallelnumericalsimulations}, using the Julia programming language~\cite{bezanson2017julia}.
The solvers are parallelized for CPUs and GPUs with \href{https://github.com/JuliaGPU/KernelAbstractions.jl}{KernelAbstractions.jl} \cite{Churavy_KernelAbstractions_jl}, which is vendor agnostic and supports both NVIDIA and AMD hardware.
For the explicit right-hand side, we use the GPU implementation of Trixi.jl, which is also based on KernelAbstractions.jl and follows the strategy of \citet{waterhouse2026gpu}, exploiting the symmetry of the two-point fluxes in the flux-differencing DG scheme.

Each kernel of the vertical solver is written once and dispatched either to a multithreaded CPU loop or to a GPU kernel, depending on where the solution array resides.
The Rosenbrock-W driver interacts with the vertical solver through two routines: \texttt{update\_jacobian!}, called once per time step, and \texttt{solve\_jacobian!}, called once per stage.
Both condensation strategies share this structure and differ only in the permutation, the resulting block sizes, and the solver for the reduced system.

\paragraph{Assembly and factorization} (\texttt{update\_jacobian!}).
The routine consists of four kernels:
\begin{enumerate}
	\item \emph{Gather:} the state is copied from the DGSEM solution array into a column-wise array, and the momentum is projected onto the vertical direction with the restriction $\mat P$ of Section~\ref{sec:jacobian_hevi}.
	\item \emph{Fill:} for each element, the frozen nodal quantities $\stateG\theta^{(e)}$ and $\stateG\alpha^{(e)}$, the diagonal and corner entries of $\mat{\mathcal A}_e$, the entries of $\mat{\mathcal B}$, $\mat{\mathcal C}$, and $\mat{\mathcal D}$, and the local Schur complement $\mat{\mathcal S}_e$ are computed; $\mat{\mathcal S}_e$ is then factorized with a dense LU decomposition, which we have implemented without pivoting.
	\item \emph{Global Schur complement:} $\mat{\mathcal S} = \mat{\mathcal D} - \mat{\mathcal C}\mat{\mathcal A}^{-1}\mat{\mathcal B}$ is assembled column by column, where the action of $\mat{\mathcal A}_e^{-1}$ is applied through the factorized $\mat{\mathcal S}_e$ and the diagonal blocks $\mat{\mathcal A}_{11}$ and $\mat{\mathcal A}_{22}$.
	\item \emph{Factorization of $\mat{\mathcal S}$:} for the two-sided strategy, $\mat{\mathcal S}$ is stored in banded format with seven bands and factorized with a banded LU decomposition (Section~\ref{sec:twosided}); for the one-sided strategy, $\mat{\mathcal S}$ is stored as a block-tridiagonal matrix with $3\times3$ diagonal and $3\times2$ off-diagonal blocks, and factorized with the specialized block Thomas algorithm of Section~\ref{sec:onesided}.
	Neither factorization uses pivoting.
\end{enumerate}
The dense blocks $\mat{\mathcal A}_{13}$, $\mat{\mathcal A}_{23}$, and $\mat{\mathcal A}_{32}$ and the corresponding entries of $\mat{\mathcal B}$ and $\mat{\mathcal C}$ are not stored: they are products of the LGL derivative matrix with diagonal matrices, and are rebuilt on the fly from the derivative matrix and the two nodal vectors $\stateG\theta^{(e)}$ and $\stateG\alpha^{(e)}$.
Since the kernels are memory bound, the additional arithmetic is cheaper than the memory traffic it saves.

\paragraph{Linear solve} (\texttt{solve\_jacobian!}).
Each stage solve \eqref{eq:vertical_stage} follows the Schur-complement solution procedure of Eqs.~\eqref{eq:schurgeneric}--\eqref{eq:schurgeneric2} and consists of five kernels:
\begin{enumerate}
	\item \emph{Gather:} restriction $\mat P\state b_i$ of the stage right-hand side.
	\item \emph{Forward elimination:} computation of the reduced right-hand side $\state g - \mat{\mathcal C}\mat{\mathcal A}^{-1}\state f$ with element-local solves.
	\item \emph{Reduced solve:} forward and backward substitution with the factorized $\mat{\mathcal S}$ (banded or block-tridiagonal) to obtain the edge unknowns $\state y$.
	\item \emph{Back-substitution:} recovery of the local unknowns $\state x = \mat{\mathcal A}^{-1}(\state f - \mat{\mathcal B}\state y)$, element by element.
	\item \emph{Scatter:} lift and scaling \eqref{eq:vertical_lift} of the result back to the DGSEM solution array.
\end{enumerate}

\paragraph{Parallelization strategy.}
The vertical columns are fully independent, so the natural unit of parallelism is the column; in distributed-memory runs, the mesh is partitioned such that no column is split across MPI ranks.
The gather and scatter kernels use one GPU thread per node.
The factorization of $\mat{\mathcal S}$ and the reduced solve are sequential recurrences in the vertical direction and therefore use one thread per column, which loops over the elements of the column; the same holds for the assembly of $\mat{\mathcal S}$, whose contributions overlap between neighboring elements.
The forward elimination and the back-substitution use one thread per element, since the element-local solves are independent.
The two strategies differ in the remaining details.
In the two-sided strategy, the rows of $\mat{\mathcal C}$ associated with element $e$ only involve local unknowns of element $e$, so each thread of the forward elimination writes only to its own entries of the reduced right-hand side; the fill kernel, in contrast, writes the direct interface coupling of $\mat{\mathcal D}$ into the rows of the neighboring element, and is therefore executed with one thread per column.
In the one-sided strategy, the local unknowns at node $1$ of element $e+1$ also contribute to the reduced right-hand side of element $e$; these contributions are written into a separate array, with a single writer per entry, and added in a subsequent kernel, which also allows one thread per element for the fill kernel.
Hence, no atomic operations are needed in either implementation.
All arrays owned by the solver are stored with the column index as the fastest-varying index, so that neighboring GPU threads, which process neighboring columns, access contiguous memory (coalesced access).
Small per-thread work arrays are kept in registers as statically sized arrays.
Finally, the Jacobian and its factorization are computed once per time step and reused in all stages of the Rosenbrock-W method.

\section{Numerical results}
\label{sec:results}
In this section we validate the scheme on a set of standard tests.
The two orderings of the vertical unknowns, one-sided vs two-sided, solve the same linear system and give the same solution to round-off.
Unless stated otherwise, the simulations are carried out using the Jacobian assuming the two-sided ordering. 
On the equiangular cubed sphere \cite{sadourny1972conservative,ronchi1996cubed} we measure the horizontal resolution by the mean node spacing along the equator,
\begin{equation}
    \Delta x \approx \frac{2 \pi \mathcal{R}}{4 K_h (\polydeg+1)},
\end{equation}
where $\mathcal{R}$ is the radius of the sphere ($\mathcal{R}_E$, or $\mathcal{R}_E/X$ on the small Earth), $K_h$ is the number of elements along each edge of a panel and $\polydeg$ is the polynomial degree.
We first verify the order of accuracy on the inertia-gravity waves in a two-dimensional channel. 
On the sphere we then run the baroclinic instability without orography, the baroclinic wave triggered by mountains, where we also check the conservation of the linear invariants with topography, and finally the gap flow and the vortex shedding of DCMIP-2025~\cite{andrews2026} on a reduced-radius sphere.
The plots are drawn by evaluating each field with the DG polynomial of each element on a uniform subgrid and all figures are produced with Makie.jl~\cite{danisch2021makie}.

\subsection{Inertia-gravity waves}
We verify the spatial order of accuracy of the DGSEM HEVI scheme on the linear inertia-gravity waves of Baldauf and Brdar~\cite{baldauf2013}, by means of an analytical solution of the linearized equations.
A channel of length $L = \SI{300}{km}$ and height $H = \SI{10}{km}$ is initialized with an isothermal atmosphere, $T_0 = \SI{250}{K}$, with a uniform background wind $u_0 = \SI{20}{m/s}$, in hydrostatic balance with the geopotential $\phi = g z$, $g = \SI{9.81}{m/s^2}$. The temperature is perturbed by
\begin{equation}
    T' = \Delta T \sin \left (\frac{ \pi z}{H} \right ) \exp \left (-\frac{(x - x_c)^2}{a^2}\right ),
\end{equation}
with $x_c = \SI{100}{km}$ and $a = \SI{5}{km}$, and the solution is advanced to $T_f = \SI{1800}{s}$. The amplitude $\Delta T = \SI{1.0e-3}{K}$ keeps the nonlinear terms small, to avoid saturation errors due to the predominance of nonlinear effects~\cite{blaise2016, baldauf2021horizontally,waruszewski2022entropy}.
The domain is discretized with $K_x \times K_z$ elements at a fixed ratio $\Delta x / \Delta z = 3$ with $\Delta x$ from $\SI{10}{km}$ to $\SI{1.25}{km}$, for polynomial degree $\polydeg \in \{ 2, 3, 4 \}$.
The time step is $\Delta t = \SI{0.1}{s}$ at $\Delta x = \SI{10}{km}$ and is scaled with $\Delta x$.
Figure~\ref{fig:igw} shows the $L_2$ errors of the horizontal and vertical velocity, of the temperature and of the pressure as functions of $\Delta x$.
The scheme converges with the expected rate $\polydeg+1$ for all variables.
At $\polydeg = 4$ the errors on the finest mesh flatten as expected, due to the nonlinear effects, as also observed and reported in~\cite{blaise2016,baldauf2021horizontally,waruszewski2022entropy, artiano2025structure, artiano2026entropy}.

\begin{figure}
    \centering
    \includegraphics[width = \textwidth]{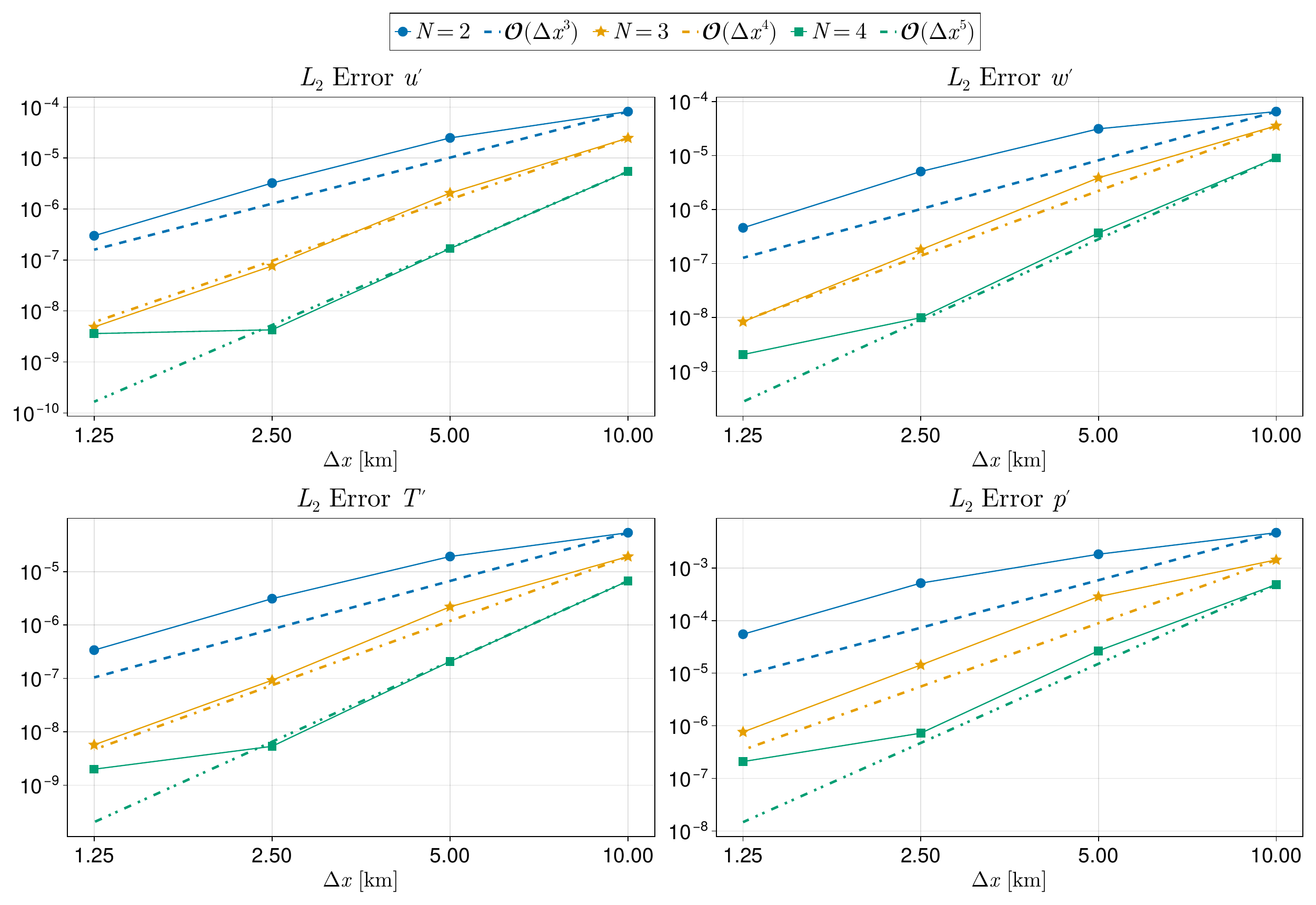}
    \caption{$L_2$ error of the horizontal velocity (top left), vertical velocity (top right), temperature (bottom left) and pressure (bottom right) as a function of the mesh size, for polynomial degrees $\polydeg \in \{2,3,4\}$.}
    \label{fig:igw}
\end{figure}

\subsection{Baroclinic instability}
\label{sec:results_baroclinic}
The first test on the sphere is the baroclinic instability of Ullrich et al.~\cite{ullrich2014proposed}, adopted in the 2016 Dynamical Core Model Intercomparison Project (DCMIP)~\cite{ullrich2017dcmip2016}.
The initial state is a balanced, zonally symmetric solution of the deep-atmosphere equations, to which a localized Gaussian perturbation of the horizontal wind is added in the northern midlatitudes; the perturbation grows into a baroclinic wave, which remains close to linear until about day~7 and then breaks.
The cubed-sphere mesh has $K_h = 20$ elements along each edge of the six panels, $K_v = 13$ elements between the surface and the model top at $\SI{30}{km}$, and $\polydeg = 4$, which gives a horizontal resolution of about $\SI{100}{km}$; the time step is $\Delta t = \SI{85}{s}$.
Figure~\ref{fig:baroclinic_p4} shows the surface pressure and the temperature and relative vorticity at $\SI{850}{hPa}$ on days 8 and 10.
The fields agree with those of the MCore and ENDGame models~\cite{ullrich2014proposed}, of MPAS~\cite{skamarock2021fully} and of the entropy-stable DG schemes of~\cite{waruszewski2022entropy, artiano2026entropy}.
The relative vorticity shows small-scale oscillations: the scheme has no explicit dissipation beyond that of the interface fluxes, and the vorticity is more sensitive to it than the pressure and the temperature~\cite{waruszewski2022entropy,skamarock2021fully}.
\begin{figure}
    \centering
    \includegraphics[width = \textwidth]{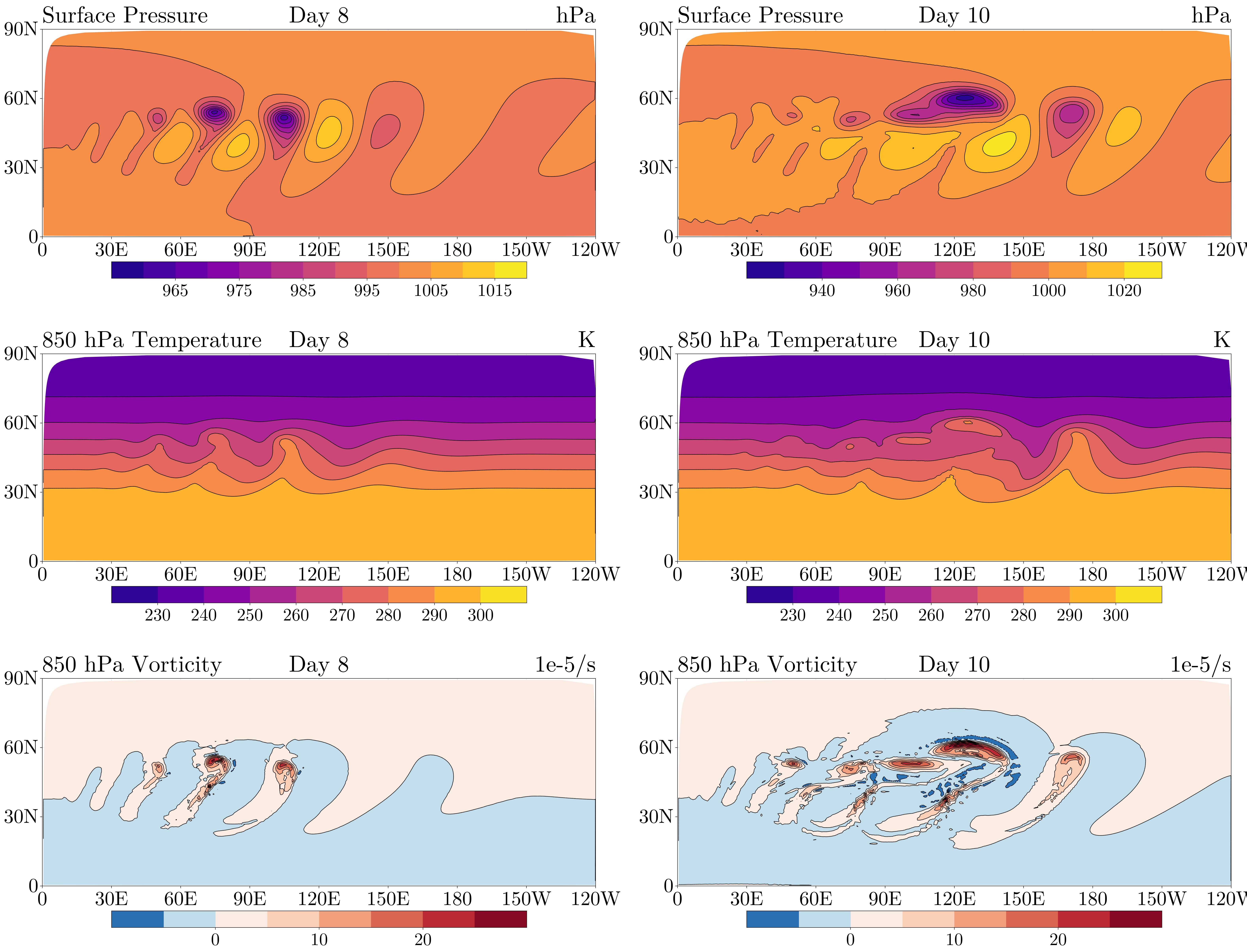}
    \caption{Surface pressure (top), \SI{850}{hPa} temperature (middle) and \SI{850}{hPa} relative vorticity (bottom) at day~8 (left) and day~10 (right) for the baroclinic instability test case, with $K_h = 20$, $K_v = 13$, and polynomial degree $\polydeg = 4$.}
    \label{fig:baroclinic_p4}
\end{figure}

\subsubsection{Performance of the dry Euler dynamical core}
\label{sec:performance}

As described in Section~\ref{sec:implementation}, our GPU implementation is vendor agnostic through KernelAbstractions.jl \cite{Churavy_KernelAbstractions_jl}, so the same code runs on NVIDIA and AMD hardware.
We measure the performance of the HEVI scheme in simulated years per day (SYPD) for the baroclinic instability of Section~\ref{sec:results_baroclinic}, using the entropy-stable and well-balanced split-form DGSEM with $\polydeg = 4$ and $K_v = 13$ vertical elements.
All runs use double precision on a single GPU, either an NVIDIA GH200 or an AMD MI210.
The SYPD quantity is assessed after running the baroclinic instability for $10$ days, excluding initialization and I/O.
The horizontal resolution at the equator ranges from $\SI{100}{km}$ to $\SI{25}{km}$, i.e., from $K_h = 20$ to $K_h = 80$ elements per panel edge.
The time step is $\Delta t = \SI{85}{s}$ at $\SI{100}{km}$ and is scaled proportionally to $\Delta x$, so that the horizontal Courant number is the same on all meshes.

Figure~\ref{fig:sypd} shows the SYPD for both GPUsm both condensation strategies and different polynomial degrees.
For the same resolutions, the performance of the method is very similar for $\polydeg = 3$ and $\polydeg = 4$, while it drops slightly for $\polydeg=5$.
On the GH200, the scheme reaches $\SI{18.8}{SYPD}$ at $\SI{100}{km}$, $\SI{2.9}{SYPD}$ at $\SI{50}{km}$, and $\SI{0.35}{SYPD}$ at $\SI{25}{km}$ with $\polydeg=4$.
The MI210 is between $2.9$ and $3.7$ times slower than the GH200, and the $\SI{25}{km}$ mesh does not fit into its $\SI{64}{GB}$ of memory.
The two condensation strategies perform similarly: on the GH200, the two-sided strategy is at most $\SI{4}{\percent}$ faster than the one-sided strategy, and on the MI210 at most $\SI{7}{\percent}$.
Since the number of vertical columns grows as $\Delta x^{-2}$ and the number of time steps as $\Delta x^{-1}$, the SYPD are expected to scale as $\Delta x^{3}$.
The measured SYPD follow this scaling closely on the finer meshes, while the decrease is slightly weaker on the coarsest meshes, presumably because these meshes do not fully occupy the GPU.

\begin{figure}
    \centering
    \includegraphics[width=\linewidth, trim=0 0 0 10, clip]{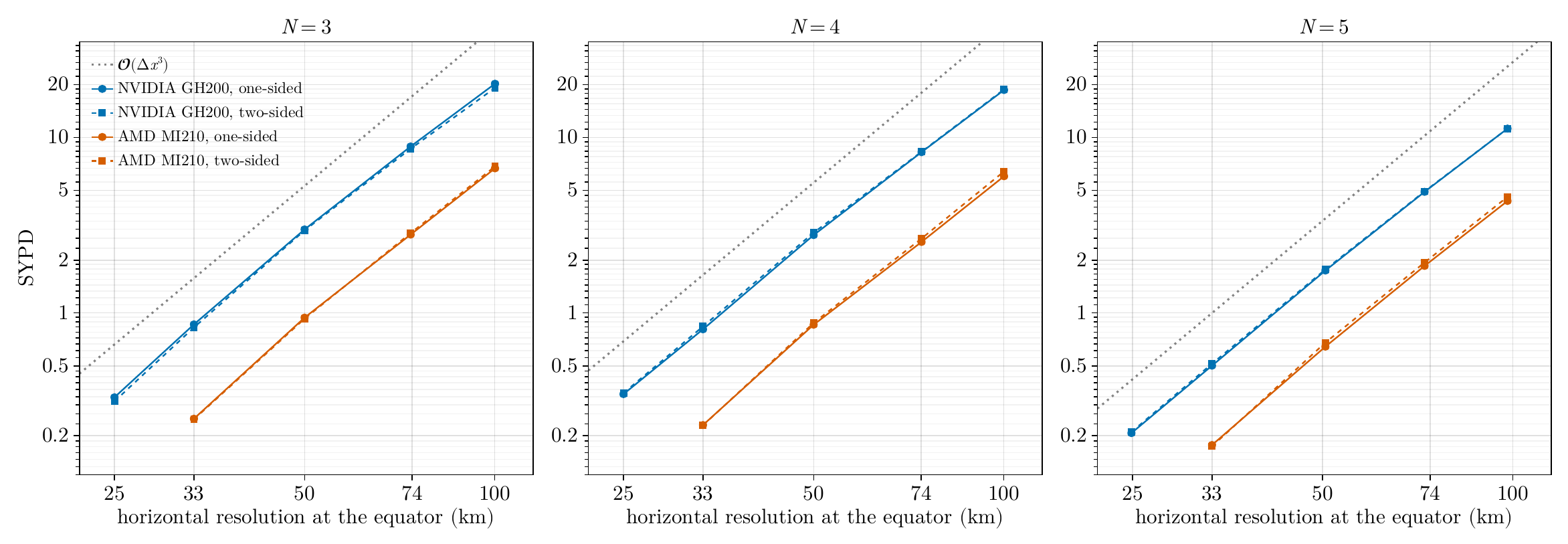}
    \caption{Simulated years per day (SYPD) for the baroclinic instability with the dry compressible Euler equations on the sphere, as a function of the horizontal resolution at the equator, for the one-sided and two-sided condensation strategies on a single NVIDIA GH200 and a single AMD MI210 GPU.
    The discretization uses $\polydeg \in \{3,4,5\}$ and $K_v = 13$, and the time step is $\Delta t = \SI{85}{s}$ at $\SI{100}{km}$ and scaled with $\Delta x$.
    The dotted line indicates the expected $\mathcal O(\Delta x^3)$ scaling.}
    \label{fig:sypd}
\end{figure}

\subsection{Mountain-induced baroclinic wave}
The mountain-induced baroclinic wave of \citet{hughes2023mountain} uses the same balanced background
state as the baroclinic instability of Section~\ref{sec:results_baroclinic}. Here,  the baroclinic waves are triggered by two mountain ridges in the northern
midlatitudes.
We run the dry variant of the test, in the deep-atmosphere form of the background state
\cite{ullrich2014proposed}, as recommended in~\cite{hughes2023mountain} for deep-atmosphere models.
The orography consists of two identical ridges centered at $45^\circ$N,
\begin{equation}
    z_s(\lambda,\varphi) = h_0 \sum_{n=1}^{2} \exp\left[-\left(\frac{\varphi-\varphi_n}{d}\right)^{6}
    - \left(\frac{l_n(\lambda)}{c}\right)^{2}\right],
    \qquad d = \frac{\Delta\varphi}{2}\left(-\ln 0.1\right)^{-1/6}, \quad
    c = \frac{\Delta\lambda}{2}\left(-\ln 0.1\right)^{-1/2},
\end{equation}
where $l_n(\lambda)$ is the longitudinal distance from the center $\lambda_n$ of the $n$-th ridge,
$h_0 = \SI{2}{km}$, $\varphi_1 = \varphi_2 = 45^\circ$, $\lambda_1 = 72^\circ$E, $\lambda_2 = 140^\circ$E,
and $\Delta\varphi = 40^\circ$ and $\Delta\lambda = 7^\circ$ are the latitudinal and longitudinal widths
at which the height drops to $10\%$ of its peak.
The initial vertical velocity is the one induced by the zonal wind along the sloping coordinate surfaces,
Eq.~(5) in~\cite{hughes2023mountain}.
The mesh has $K_h = 50$ elements per panel edge, $K_v = 13$ uniform layers with the Gal-Chen
terrain-following coordinate up to the model top at $\SI{30}{km}$, and $\polydeg = 3$, i.e., a horizontal
resolution of about $\SI{50}{km}$, as the ne60 configuration of~\cite{hughes2023mountain};
the time step is $\Delta t = \SI{40}{s}$.
Figure~\ref{fig:hj23_maps} shows the sea-level pressure and the temperature at $\SI{850}{hPa}$ on days 3 to 6.
Each ridge triggers a baroclinic wave in its lee, and the two low-pressure systems are almost identical
at day 4.
At day 5, the wave triggered by the first ridge reaches the second one, and by day 6 the orographic forcing
has visibly changed its structure with respect to the wave triggered by the second ridge, as described
in~\cite{hughes2023mountain}.
The longitude--height section at $45^\circ$N at day 4 (Figure~\ref{fig:hj23_section}) shows the warm
anomaly of the developing wave, with its maximum at the surface near $122^\circ$E, and the vertically
propagating inertia--gravity waves in the lee of the second ridge, with the same structure and amplitude
as the dry simulation of~\cite[Fig.~7]{hughes2023mountain}.
\begin{figure}
    \centering
    \includegraphics[width=1.05\textwidth, keepaspectratio]{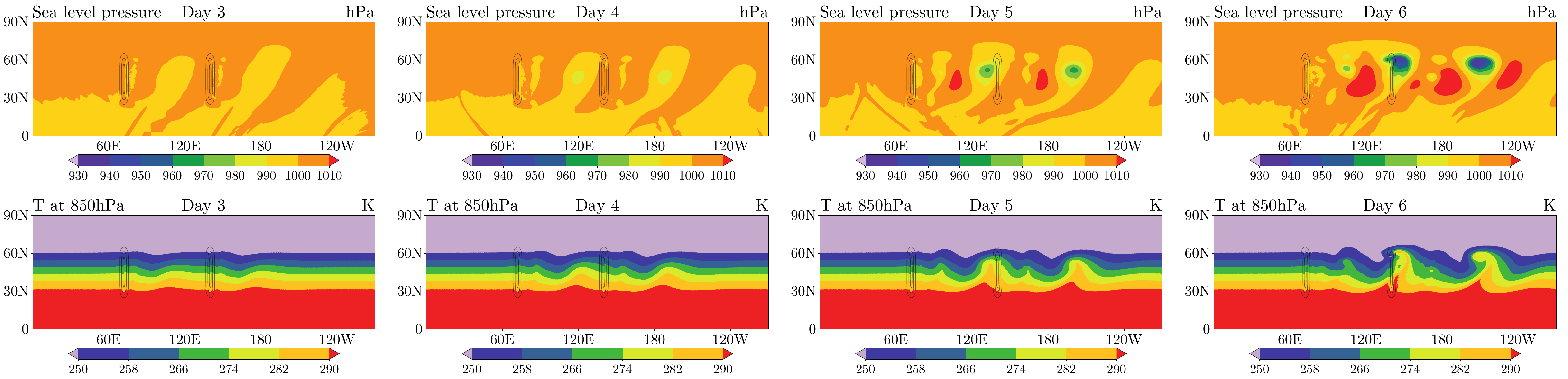}
    \caption{Mountain-induced baroclinic wave: sea-level pressure (top) and temperature at $\SI{850}{hPa}$
    (bottom) on days 3, 4, 5 and 6 (left to right), with $K_h = 50$, $K_v = 13$ and $\polydeg = 3$.
    The contours mark the ridges.}
    \label{fig:hj23_maps}
\end{figure}

\begin{figure}
    \centering
    \includegraphics[width=\textwidth]{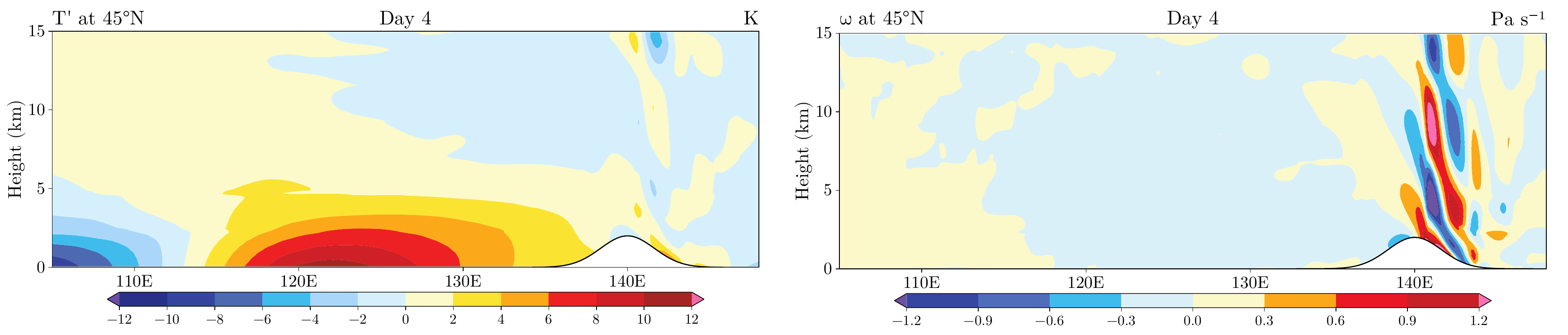}
    \caption{Mountain-induced baroclinic wave: temperature perturbation (left) and vertical pressure
    velocity $\omega$ (right) at $45^\circ$N on day 4, with $K_h = 50$, $K_v = 13$ and $\polydeg = 3$.}
    \label{fig:hj23_section}
\end{figure}
Finally, in Figure~\ref{fig:hj23_conservation}, the total mass and the total potential-temperature density on the sphere with orography are conserved to
machine precision as predicted
by Corollary~\ref{cor:hevi_invariants}.
\begin{figure}
    \centering
    \includegraphics[width=0.6\textwidth]{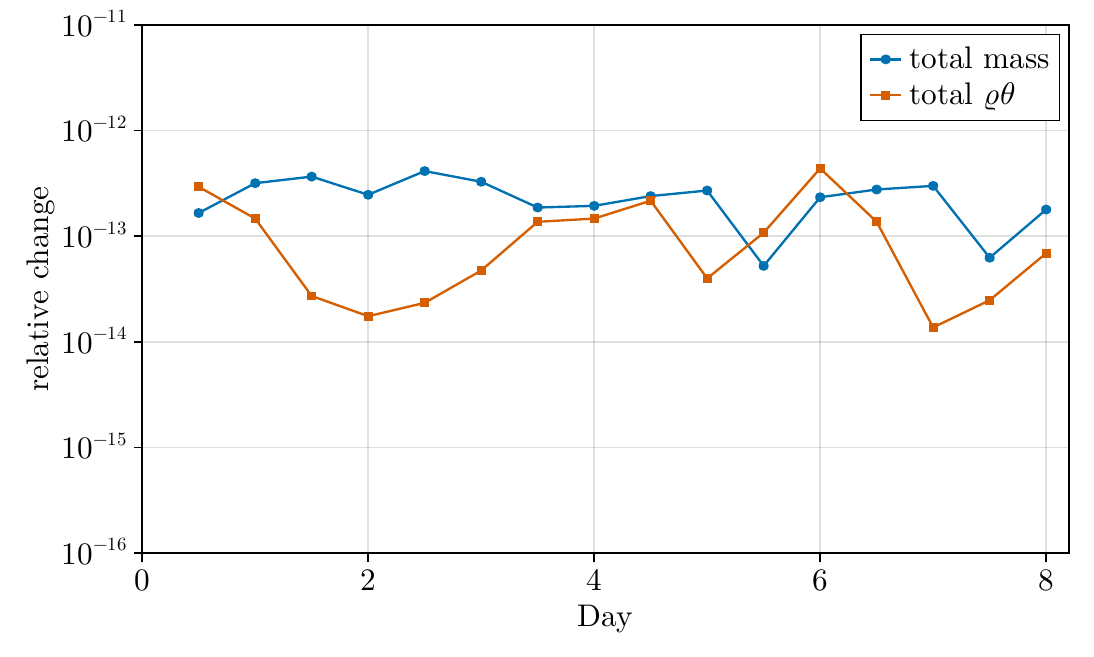}
    \caption{Mountain-induced baroclinic wave: relative change of the total mass and of the total
    potential-temperature density with respect to the initial state, with $K_h = 50$, $K_v = 13$ and
    $\polydeg = 3$.}
    \label{fig:hj23_conservation}
\end{figure}

\subsection{Gap flow}
\label{sec:gapflow}
The gap flow is one of the mountain-generated mesoscale tests of DCMIP-2025 reported in~\cite{andrews2026}.
The test case is run on a sphere of reduced radius $\mathcal{R} = \mathcal{R}_E / X$ with $\mathcal{R}_E = \SI{6371.229}{km}$ and $X = 20$, rotating with $\Omega = X\,\Omega_E$, so that the orography is steep and mesoscale dynamics appear at an affordable resolution.
The initial state is an isothermal atmosphere in hydrostatic balance, with a zonal wind in solid-body rotation, with balanced surface pressure~\cite{andrews2026},
\begin{equation}
\begin{aligned}
    u_{\lambda} &= u_0 \cos \varphi \\
    u_{\phi} &= 0, \\
    T &= T_0, \\
    p &= p_s \exp \left(-\frac{g (z - z_s)}{R_d T_0}\right),\\
    p_s &= p_{0} \exp\left[-\frac{a \mathcal{N}^2 u_0}{2 g^2 \kappa}\left(\frac{u_0}{a} + 2\Omega\right)\left(\sin^2\varphi - 1\right) - \frac{\mathcal{N}^2}{g \kappa}\, z_s\right],
\end{aligned}
\end{equation}
where $u_{\lambda}$ and $u_{\phi}$ are the zonal and meridional velocity components, $u_0 = \SI{10}{m/s}$, $T_0 = \SI{288}{K}$, $p_0 = \SI{e5}{Pa}$, $\kappa = R_d / c_p$ and the constant Brunt--V\"ais\"al\"a frequency $\mathcal{N} = \sqrt{g^2 / (c_p T_0)} \approx \SI{0.0182}{s^{-1}}$.
We initialize the vertical velocity to zero, since the nonzero vertical velocity that follows the coordinate lines, as in~\cite{andrews2026} for nonhydrostatic models, does not have noticeable influence.
The orography $z_s = f\,q$ is a meridional mountain chain cut by a gap at its center,
\begin{equation}
    f(\lambda, \varphi) = h_0 \exp\left(-\left(\frac{\lambda - \lambda_c}{d_1}\right)^{b} - \left(\frac{\varphi - \varphi_c}{d_2}\right)^{b}\right),
    \qquad
    q(\lambda, \varphi) = 1 - \exp\left(-\left(\frac{\varphi - \varphi_c}{d_3}\right)^{b}\right),
\end{equation}
with $h_0 = \SI{1500}{m}$, $(\lambda_c, \varphi_c) = (\pi, 0)$, $b = 10$ and $d_i = \frac{x_i}{2\mathcal{R}} (\ln 10)^{-1/b}$, where $x_1 = \SI{40}{km}$, $x_2 = \SI{300}{km}$ and $x_3 = \SI{50}{km}$ are the longitudinal extent of the chain, its latitudinal extent and the width of the gap.
The model top is at $\SI{20}{km}$, and a Rayleigh damping layer above $\SI{10}{km}$ relaxes the velocity toward its initial value to absorb the vertically propagating waves.
The inverse Froude number places the flow in the nonlinear flow-around regime: the wind accelerates through the gap into a jet along the equator, decelerates and even reverses on both sides of the chain, and a pair of lee vortices detaches from the orography.
The mesh has $K_h = 56$ elements per panel edge, $K_v = 19$ layers  and $\polydeg = 3$, i.e., a horizontal resolution of about $\SI{2.25}{km}$ as in the reference simulations; the time step is $\Delta t = \SI{2.4}{s}$.
Figure~\ref{fig:gapflow} shows the zonal wind perturbation $u' = (u_\lambda - u_0\cos\varphi)/(u_0\cos\varphi)$ and the temperature perturbation $T' = T - T_0$ at height $z = \SI{300}{m}$, at days 0.125, 0.25, 0.375 and 0.5.
The accelerated flow through the gap, the regions of reversed flow on both sides of the chain and the pair of lee vortices in the temperature field agree with the four dynamical cores of~\cite{andrews2026}; downstream of the gap, the temperature field shows more small-scale structure than CAM-SE and CAM-MPAS, as in CAM-FV3 and GungHo.
The grid oscillations are due to the very minimal dissipation provided by our method.

\begin{figure}
    \centering
\includegraphics[width=\textwidth, keepaspectratio]{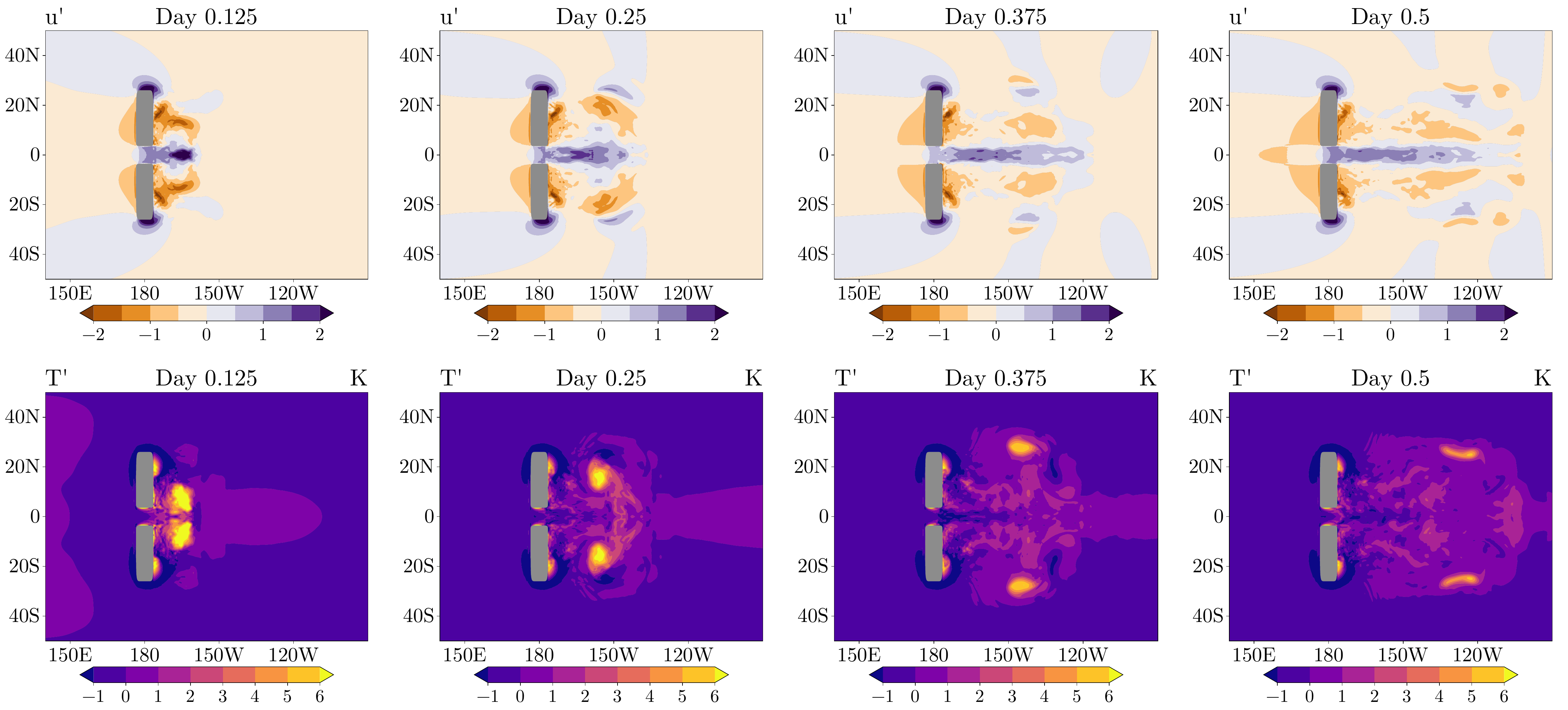}
\caption{Gap flow at $z = \SI{300}{m}$ on days 0.125, 0.25, 0.375 and 0.5 (left to right): zonal wind perturbation $u' = (u_\lambda - u_0\cos\varphi)/(u_0\cos\varphi)$ (top) and temperature perturbation $T' = T - T_0$ in K (bottom) at $\SI{2.25}{km}$ horizontal resolution scale.}
    \label{fig:gapflow}
\end{figure}

\subsection{Vortex shedding}
\label{sec:vortex_shedding}
The last test case is the vortex shedding of the same DCMIP-2025 suite reported in~\cite{andrews2026}.
The small Earth, the initial state, the model top, and the damping layer are those of the gap flow in Section~\ref{sec:gapflow}.
The difference lies in the topography, which is described here by an isolated Gaussian mountain,
\begin{equation}
    z_s(\lambda, \varphi) = h_0 \exp\left(-\left(\frac{r}{d}\right)^2\right),
    \qquad
    r = \mathcal{R} \arccos\left(\sin\varphi \sin\varphi_c + \cos\varphi \cos\varphi_c \cos(\lambda - \lambda_c)\right),
\end{equation}
where $r$ is the great-circle distance from the center $(\lambda_c, \varphi_c) = (\pi, \pi/9)$, $h_0 = \SI{2000}{m}$ and $d = \SI{12.5}{km}$.
The mesh has $K_h = 56$ elements per panel edge, $K_v = 19$ uniform layers and $\polydeg = 3$, as for the gap flow, i.e., a horizontal resolution scale of $\SI{2.25}{km}$, with a time step $\Delta t = \SI{2.4}{s}$. The simulation runs for a total duration of 1 day.
Figure~\ref{fig:vortex} shows the zonal wind perturbation and the relative vorticity at $z = \SI{300}{m}$
on days 0.25 to 1.
The structure of the wake, with alternating vortices of opposite sign and a recirculation zone behind the mountain, agrees with the four dynamical cores of~\cite{andrews2026}.
As in the gap flow, the vorticity shows small-scale structure along the shear layers, which is more pronounced than in CAM-SE and CAM-MPAS.
This is due to the low intrinsic dissipation of the scheme, since the dissipation is only introduced at the element interfaces through the entropy-stable surface fluxes.
\begin{figure}
    \centering
\includegraphics[width = 1.0\textwidth, keepaspectratio]{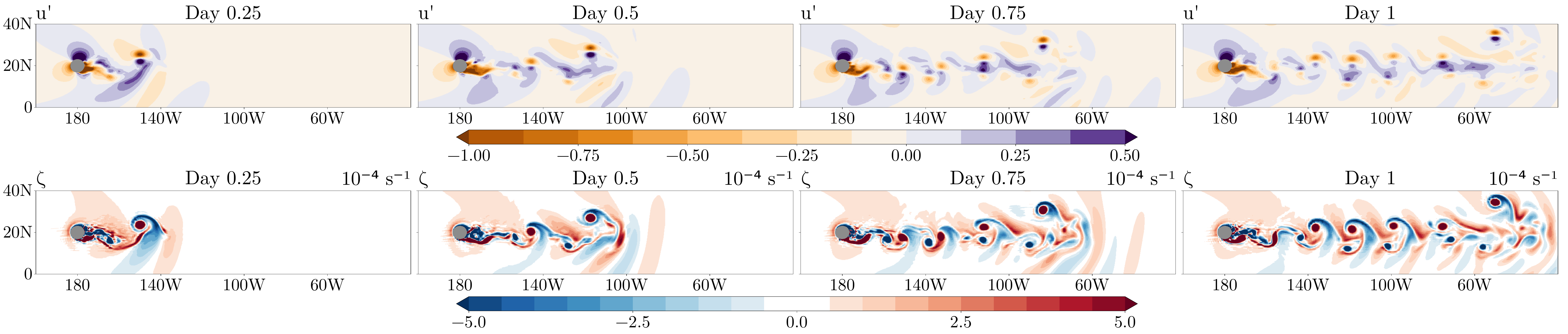}
    \caption{Vortex shedding at $z = \SI{300}{m}$ on days 0.25, 0.5, 0.75 and 1 (left to right), with
rotation: zonal wind perturbation $u' = (u_\lambda - u_0\cos\varphi)/(u_0\cos\varphi)$ (top) and relative vorticity
$\zeta$ in $\SI{e-4}{s^{-1}}$ (bottom) at $\SI{2.25}{km}$ horizontal resolution scale.}
    \label{fig:vortex}
\end{figure}

\section{Conclusions} \label{sec:Conclusions}

We have presented nested Schur-complement solvers for the vertical linear systems that arise in the horizontally explicit, vertically implicit (HEVI) time integration of a split-form LGL-DGSEM for the compressible Euler equations with gravity.
Exploiting the freedom of Rosenbrock-W methods in the choice of the Jacobian, we derived an approximate Jacobian from the standard strong-form DGSEM that retains only the vertically propagating acoustic waves and gravity, decouples into independent column systems, is free of geometric (Christoffel) terms on curvilinear meshes with orography, and preserves the conservation of total mass and potential-temperature density.
The solvers combine a local Schur complement, which analytically eliminates the density and potential-temperature density within each element, with a global static condensation onto element-edge unknowns.
Of the two proposed condensation strategies, the \textit{two-sided} one leads to a banded reduced system of size $4N_z$, and the \textit{one-sided} one to a smaller block-tridiagonal system of size $3N_z$ that is solved with a specialized block Thomas algorithm.
Both strategies are pivot-free and implemented for CPUs and GPUs in a vendor-agnostic way in Trixi.jl and TrixiAtmo.jl, and they are combined with a new three-stage, second-order, L-stable Rosenbrock-W method that reduces to the SSP Runge--Kutta method of \citet{shu1988efficient} for the explicitly treated components.

We have tested the accuracy and performance of the scheme on standard test cases from the literature: the channel-confined inertia-gravity waves, the baroclinic wave, a mountain-induced baroclinic instability, and the DCMIP-2025 gap-flow and vortex-shedding test cases.

\section*{Acknowledgments}

Andrés M. Rueda-Ramírez gratefully acknowledges funding from the Spanish Ministry of Science, Innovation, and Universities through the ``Beatriz Galindo'' grant (BG23-00062).
This project has received funding from the European Research Council (ERC) under the European Union's Horizon Europe research and innovation program (grant agreement No. 101167322 - TRANSDIFFUSE).
Gregor Gassner and Andrés M. Rueda-Ramírez acknowledge funding through the German Federal Ministry for Education and Research (BMBF) project ``ICON-DG'' (01LK2315B) of the ``WarmWorld Smarter'' program.

Marco Artiano and Hendrik Ranocha were supported by the Deutsche Forschungsgemeinschaft (DFG, German Research Foundation) through project number 528753982 and within the DFG priority program SPP~2410 through project number 526031774.
This work was also supported by the Max Planck Graduate Center with the Johannes Gutenberg University of Mainz (MPGC) and the Mainz Institute of Multiscale Modeling (M\textsuperscript{3}ODEL).

\section*{Tool disclosure}

Claude Opus 5 was used to assist software development.

\printcredits

\bibliographystyle{model1-num-names-doi}

\bibliography{Biblio.bib}

\section*{Appendices}
\appendix

\section{Construction of an L-stable Rosenbrock-W method} \label{app:tableau}

In what follows, we derive the tableau for a Rosenbrock-W method \eqref{eq:rosenbrockW},
\begin{equation}\label{eq:rosenbrockW2}
	\left( \frac{1}{\hat \gamma \Delta t} \mat{I} - \mat{J} \right) \state{K}_i = \state{R}\Big(\state{U}^n + \textstyle\sum_{j<i} \tilde{a}_{ij} \state{K}_j\Big) + \frac{1}{\Delta t}\textstyle\sum_{j<i} c_{ij} \state K_j,
	\qquad i = 1, \dots, s,
\end{equation}
which is L-stable, has a constant $\hat \gamma$ for all stages, and reduces to a strong stability-preserving Runge--Kutta scheme in the purely explicit setting.

We use the row sums $a_i = \sum_{j<i}a_{ij}$ and $\gamma_i = \sum_{j\le i}\gamma_{ij}$, where the latter includes the diagonal entry of $\mat\Gamma$.
We derive a three-stage method whose explicit part, i.e., the method obtained for $\mat J = \mat 0$, is the three-stage, third-order SSP Runge--Kutta method of \citet{shu1988efficient},
\begin{equation}\label{eq:ssprk3}
	a_{21} = 1, \quad a_{31} = a_{32} = \tfrac14, \quad \state b = \big(\tfrac16, \tfrac16, \tfrac23\big)^T.
\end{equation}
For $\mat J = \mat 0$, \eqref{eq:rosenbrockW}--\eqref{eq:Rosenbrock_update} reduces exactly to this method, irrespective of $\mat\Gamma$.
In the HEVI setting, the same holds for all components on which $\mat J$ vanishes, e.g., passive tracers, which are therefore advanced with an SSP method that preserves the stability properties of the spatial discretization under a time-step restriction \cite{shu1988efficient,gottlieb2001strong}.

Since \eqref{eq:ssprk3} satisfies all order conditions of explicit Runge--Kutta methods up to order three, only the order conditions that involve $\mat J$ restrict the four remaining parameters $\hat\gamma$, $\gamma_{21}$, $\gamma_{31}$, and $\gamma_{32}$.
For a W-method, i.e., for an arbitrary $\mat J$, these conditions read \cite{steihaug1979attempt,hairer1996solving}
\begin{align}
	&\text{order 2:} && \textstyle\sum_{i} b_i\gamma_i = 0, \label{eq:W_order2}\\
	&\text{order 3:} && \textstyle\sum_{i,j} b_ia_{ij}\gamma_j = 0, \qquad
	\textstyle\sum_{i,j} b_i\gamma_{ij}a_j = 0, \qquad
	\textstyle\sum_{i,j} b_i\gamma_{ij}\gamma_j = 0, \label{eq:W_order3}
\end{align}
where the three conditions of \eqref{eq:W_order3} correspond to the elementary differentials $\state R'\mat J\state R$, $\mat J\state R'\state R$, and $\mat J^2\state R$, respectively.
With \eqref{eq:ssprk3}, the first two conditions of \eqref{eq:W_order3} yield $\gamma_{21} = -3\hat\gamma$ and $\gamma_{32} = -\tfrac34\hat\gamma$, \eqref{eq:W_order2} then yields $\gamma_{31} = 0$, and the last condition of \eqref{eq:W_order3} reduces to $\hat\gamma^2/2 = 0$.
Hence, there is no third-order Rosenbrock-W method with a nonzero constant diagonal whose explicit part is \eqref{eq:ssprk3}, and we impose only the second-order condition \eqref{eq:W_order2}, which, written out, reads
\begin{equation}\label{eq:2ndorder_rosw}
    6\hat\gamma + \gamma_{21} + 4\left(\gamma_{31} + \gamma_{32}\right) = 0.
\end{equation}

The remaining free parameters are chosen to obtain an L-stable method.
Following \cite{steihaug1979attempt,hairer1996solving,rangangermann2005}, we study the linear stability of the W-method for the scalar test equation $\dot u = \lambda u$ with the exact Jacobian $\mat J = \lambda$, which yields the stability function
\begin{equation}
	R(z) = 1 + z\,\state b^T\big(\mat I - z(\mat a + \mat\Gamma)\big)^{-1}\state 1, \qquad z = \lambda\Delta t.
\end{equation}
Using \eqref{eq:2ndorder_rosw}, its limit at infinity is
\begin{equation}
	R(\infty) = -\frac{(1+\gamma_{21})(1+4\gamma_{32}) - 3\hat\gamma\,(2\hat\gamma^2 - 4\hat\gamma + 1)}{6\hat\gamma^3}.
\end{equation}
Hence, $R(\infty) = 0$ if and only if
\begin{equation}\label{eq:Lstability_rosw}
    (1 + \gamma_{21}) (1 + 4 \gamma_{32}) = 3 \hat{\gamma} (2\hat{\gamma}^2 -4 \hat{\gamma}+1).
\end{equation}
Since L-stability requires A-stability in addition to $R(\infty) = 0$, we first choose the remaining parameters such that \eqref{eq:Lstability_rosw} holds, and then verify A-stability for the resulting coefficients.
We satisfy \eqref{eq:Lstability_rosw} by requiring that both of its sides vanish.
The right-hand side vanishes for $\hat\gamma = 1 \pm {1}/{\sqrt2}$, and we choose $\hat{\gamma} = 1 - {1}/{\sqrt{2}}$.
Instead of setting one of the factors on the left-hand side to zero directly, we fix $\gamma_{21}$ by imposing the first third-order condition of \eqref{eq:W_order3}, $\sum_{i,j} b_ia_{ij}\gamma_j = 0$, which cancels the leading error term associated with $\state R'\mat J\state R$ and yields $\gamma_{21} = -3\hat\gamma = {3}/{\sqrt2} - 3$.
Since $1 + \gamma_{21} \neq 0$, \eqref{eq:Lstability_rosw} then requires $\gamma_{32} = -\tfrac14$, and \eqref{eq:2ndorder_rosw} gives 
$$\gamma_{31} = \tfrac{3\sqrt2 - 4}{8}.
$$
This particular choice of $\hat\gamma$ and $\gamma_{21}$ performed best in our experiments, resulting in a slightly larger stability region.

It remains now to show that the method is A-stable.
Since $\mat a + \mat\Gamma$ is lower triangular with constant diagonal $\hat\gamma$, the stability function coincides with that of implicit RK method with coefficient matrix $\mat a + \mat\Gamma$ and weights $\state b$.
Hence, it is a rational function $R(z) = P(z)/Q(z)$, whose denominator is $Q(z) = \det\big(\mat I - z(\mat a + \mat\Gamma)\big) = (1-\hat\gamma z)^3$ and whose numerator is $P(z) = \det \big(\mat I - z(\mat a + \mat\Gamma) + z\, \state 1 \state b^T\big)$\cite[Sec-IV.3]{hairer1996solving}.
For the chosen coefficients, we obtain
\begin{equation}\label{eq:R_closed}
	R(z) = \frac{P(z)}{Q(z)} = \frac{1 + (1 - 3 \hat{\gamma} )\, z\,(1 - z)}{(1-\hat\gamma z)^3}.
\end{equation}
Following \cite[Sec.~IV.3]{hairer1996solving}, a rational function $R = P/Q$ is A-stable if and only if
\begin{equation}
        E(y) = |Q(\mathrm i y)|^2 - |P(\mathrm i y)|^2 \ge 0,\quad \forall y \in \mathbb{R},
\end{equation}
 and $R(z)$ is analytic for  $\operatorname{Re}(z) < 0$.
For \eqref{eq:R_closed}, the E-polynomial reads
\begin{equation}
	E(y) = \hat\gamma^4 y^4 + \hat\gamma^6 y^6 \geq 0, \qquad \forall y \in \mathbb R,
\end{equation}
Moreover, the only pole of $R$ is $z = 1/\hat\gamma > 0$, so $R$ is analytic in the open left half-plane. 
Hence, the method is A-stable and, since $R(\infty) = 0$ by construction, L-stable.
The resulting coefficients of both forms are listed in Table~\ref{tab:ssp_tableau}.

\section{Preservation of linear invariants}
\label{app:linear_invariants}

We say that a vector $\state W$ defines a \emph{linear invariant} of the semi-discrete system $\dot{\state U} = \state R(\state U)$ if
\begin{equation}\label{eq:linear_invariant}
	\state W^T \state R(\state U) = 0 \qquad \text{for all } \state U .
\end{equation}
Then $\frac{\mathrm d}{\mathrm dt}\,\state W^T\state U = 0$, i.e., the quantity $\state W^T\state U$ is conserved by the semi-discretization.
For a conservative DGSEM discretization with periodic or impermeable boundaries, the total mass and the total potential-temperature density are linear invariants, where $\state W$ contains the quadrature weights multiplied by the mapping Jacobian at the corresponding degrees of freedom, and zeros elsewhere.

\begin{lemma}\label{lem:rosenbrock_invariants}
Let $\state W$ define a linear invariant of $\dot{\state U} = \state R(\state U)$, and let the approximate Jacobian satisfy
\begin{equation}\label{eq:jacobian_invariant}
	\state W^T \mat J = \mat 0 .
\end{equation}
Then the Rosenbrock-W method \eqref{eq:rosenbrockW}--\eqref{eq:Rosenbrock_update} preserves the linear invariant, i.e., $\state W^T\state U^{n+1} = \state W^T\state U^n$, provided that the stage equations are solved exactly.
\end{lemma}

\begin{proof}
We show by induction over the stages that $\state W^T\state K_i = 0$ for all $i$.
Multiplying the stage equation \eqref{eq:rosenbrockW} from the left by $\state W^T$ and using \eqref{eq:jacobian_invariant} gives
\begin{equation*}
	\frac{1}{\hat\gamma\Delta t}\,\state W^T\state K_i
	= \state W^T\state R\Big(\state U^n + \textstyle\sum_{j<i} \tilde{a}_{ij}\state K_j\Big)
	+ \frac{1}{\Delta t}\sum_{j<i} c_{ij}\,\state W^T\state K_j.
\end{equation*}
The first term on the right-hand side vanishes by \eqref{eq:linear_invariant}, irrespective of the argument of $\state R$.
For $i = 1$, the sum is empty, so $\state W^T\state K_1 = 0$.
If $\state W^T\state K_j = 0$ for all $j < i$, the sum vanishes, and hence $\state W^T\state K_i = 0$.
Finally, \eqref{eq:Rosenbrock_update} yields $\state W^T\state U^{n+1} = \state W^T\state U^n + \sum_i m_i\,\state W^T\state K_i = \state W^T\state U^n$.
\end{proof}

Condition \eqref{eq:jacobian_invariant} is automatically satisfied by the exact Jacobian, since differentiating \eqref{eq:linear_invariant} with respect to $\state U$ yields $\state W^T\,\partial\state R/\partial\state U = \mat 0$.
For a W-method, however, it is a property of the chosen approximation {and metric terms} and has to be verified; if it is violated, $\state W^T\state K_i = \hat\gamma\Delta t\,\state W^T\mat J\,\state K_i \neq 0$ in general, and the invariant is not preserved.
We now show that the approximate Jacobian of Section~\ref{sec:jacobian_hevi} satisfies \eqref{eq:jacobian_invariant} for the total mass and the total potential-temperature density.

\begin{lemma}\label{lem:vertical_invariants}
Consider one vertical column with reflecting walls at the bottom and top, at which the numerical fluxes of $\varrho$ and $\varrho\theta$ vanish.
Let $\state W_\varrho$ be the vector with entries $J^{(e)}_k\omega_k$ at the density degrees of freedom of the column, ordered as in \eqref{eq:orderingJac}, and zeros elsewhere, and define $\state W_{\varrho\theta}$ analogously for the potential-temperature density.
Then
\begin{equation}
	\state W_\varrho^T\mat J^V = \mat 0, \qquad \state W_{\varrho\theta}^T\mat J^V = \mat 0 .
\end{equation}
\end{lemma}

\begin{proof}
Since $\mat J^V = \partial\state R^V/\partial\state U^V$ with the frozen coefficients of Section~\ref{sec:jacobian_hevi}, it suffices to show that $\state W_\varrho^T\state R^V(\state U^V) = 0$ and $\state W_{\varrho\theta}^T\state R^V(\state U^V) = 0$ for all $\state U^V$.
We consider the density; the argument for the potential-temperature density is identical.
Denote by $\contSt f^{\varrho}_m = n^{(e)}_m(\varrho\tilde w)^{(e)}_m$ the nodal flux of the density equation.
By \eqref{eq:vertical_residual},
\begin{equation*}
	\state W_\varrho^T\state R^V
	= \sum_{e}\sum_{k=0}^{\polydeg} J^{(e)}_k\omega_k\, R^{\varrho,(e)}_k
	= -\sum_{e}\left[\sum_{k=0}^{\polydeg}\sum_{m=0}^{\polydeg} \hat D_{km}\,\contSt f^{\varrho}_m + \numflux f^{\varrho}_{e+\frac12} - \numflux f^{\varrho}_{e-\frac12}\right].
\end{equation*}
The volume term vanishes, because the columns of $\hat{\mat D} = \mat Q - \mat B$ sum to zero: since the differentiation matrix is exact for constants, $\mat Q\state 1 = \mat M\mat D\state 1 = \state 0$, and the SBP property $\mat Q + \mat Q^T = \mat B$ yields $\state 1^T\mat Q = (\mat Q^T\state 1)^T = (\mat B\state 1 - \mat Q\state 1)^T = \state 1^T\mat B$, i.e., $\state 1^T\hat{\mat D} = \state 0^T$.
The surface terms telescope, since the surface numerical flux and the contravariant metric terms at each interior interface are single-valued, leaving only the fluxes at the bottom and top walls, which vanish by assumption.
Hence $\state W_\varrho^T\state R^V = 0$ for all $\state U^V$, and differentiation yields $\state W_\varrho^T\mat J^V = \mat 0$.
For the potential-temperature density, the same argument applies with the nodal flux $n^{(e)}_m\theta^{(e)}_{f,m}(\varrho\tilde w)^{(e)}_m$; the numerical flux $\numflux f^{\varrho\theta}$ is single-valued because its dissipative term uses the interface average $\avg{\theta}$ rather than the element-local values of $\theta_f$.
\end{proof}

\begin{remark}
    Note that the conservation of linear invariants here is guaranteed by the particular choice of the discrete metric terms of the Jacobian approximation, which is derived from the DGSEM semidiscretization.
    In general, a consistent and locally conservative discretization of the Jacobian does not guarantee the conservation of linear invariants.
\end{remark}

\begin{corollary}\label{cor:hevi_invariants}
Assume that the split-form DGSEM discretization conserves the total mass and the total potential-temperature density, i.e., that the corresponding vectors $\state W$ define linear invariants of $\dot{\state U} = \state R(\state U)$.
Then the HEVI Rosenbrock-W scheme with the approximate Jacobian $\mat J = \mat L\,\mat J^V\mat P$ of Section~\ref{sec:jacobian_hevi} preserves both quantities, provided that the column systems \eqref{eq:vertical_stage} are solved exactly.
\end{corollary}

\begin{proof}
Let $\state W$ define the total mass, and let $\state W_V$ collect the corresponding weights of all columns in the ordering of $\mat J^V$, i.e., the vectors $\state W_\varrho$ of Lemma~\ref{lem:vertical_invariants}, scaled by the horizontal quadrature weights of each column.
The lift $\mat L$ of \eqref{eq:restriction_lift} maps the density component of the vertical subsystem to the density component of the full state without modification, and $\state W$ vanishes at all momentum degrees of freedom.
Hence $\state W^T\mat L = \state W_V^T$, and Lemma~\ref{lem:vertical_invariants}, applied column by column, yields
\begin{equation*}
	\state W^T\mat J = \state W^T\mat L\,\mat J^V\mat P = \state W_V^T\mat J^V\mat P = \mat 0 .
\end{equation*}
Lemma~\ref{lem:rosenbrock_invariants} then gives the result.
The same argument applies to the total potential-temperature density.
\end{proof}

In finite-precision arithmetic, the invariants are preserved up to round-off errors, since the column systems are solved with direct methods.
Note that the argument of Lemma~\ref{lem:vertical_invariants} relies on the fact that $\mat J^V$ is the exact Jacobian of a conservative (linearized) vertical residual.
If the vertical subsystem is formulated in other variables through a node-wise change of variables, as discussed in Remark~\ref{rem:other_formulations}, the resulting approximate Jacobian combines the rows of $\mat J^V$ with state-dependent coefficients and does not satisfy \eqref{eq:jacobian_invariant} for the thermodynamic variable of the dynamical core in general.

\end{document}